\documentclass[11pt,a4paper]{amsart}
\usepackage{amsmath,amsthm,amssymb,a4wide}
\usepackage[utf8]{inputenc}
\usepackage{graphicx}
\usepackage{latexsym}
\usepackage{epsfig}
\usepackage{enumerate}

\usepackage{caption}
\usepackage{subcaption}

\usepackage{hyperref}

\newtheorem{thm}{Theorem}[section]
\newtheorem{prop}[thm]{Proposition}
\newtheorem{lemma}[thm]{Lemma}
\newtheorem{cor}[thm]{Corollary}
\newtheorem{rem}[thm]{Remark}

\theoremstyle{definition}

\numberwithin{equation}{section}

\newcommand{\R}{\mathbb{R}}
\newcommand{\RR}{\mathbb{R}}

\newcommand{\M}{{\mathcal M}}

\newcommand{\eps}{\varepsilon}

\def\H{\mathcal{H}}
\def\M{\mathcal{M}}
\def\U{\mathcal{U}}
\def\d{\, {\rm d}}
\def\Ue{\U^\eps}
\def\Me{\M^\eps}

\def\dH{\partial H}

\begin{document}

\title[Diffuse Interface Methods for Inverse Problems]{Diffuse Interface Methods for Inverse Problems: \\ Case Study for an Elliptic  Cauchy Problem}
\author[M. Burger]{Martin Burger$^{\dag,+}$}
\author[O. L. Elvetun]{Ole Løseth Elvetun$^*$}
\author[M. Schlottbom]{Matthias Schlottbom$^{\dag,\circ}$}
\thanks{$^\dag$ Institute for Computational and Applied Mathematics,
University of Münster, Einsteinstr. 62, 48149 M\"unster, Germany.\\
$^+$ Cells in Motion Cluster of Excellence, University of Münster. \\
$^*$ Dept. of Mathematical Sciences and Technology, Norwegian University of Life Sciences\\
$^\circ$ Corresponding author\\
Email: {\tt $\{$burger,schlottbom$\}$@wwu.de, ole.elvetun@nmbu.no}
}
\date{\today}

\begin{abstract}
Many inverse problems have to deal with complex, evolving and often not exactly known geometries, e.g. as domains of forward problems modeled by partial differential equations. This makes it desirable to use methods which are robust with respect to perturbed or not well resolved domains, and which allow for efficient discretizations not resolving any fine detail of those geometries. For forward problems in partial differential equations methods based on diffuse interface representations gained strong attention in the last years, but so far they have not been considered systematically for inverse problems. In this work we introduce a diffuse domain method as a tool for the solution of variational inverse problems. As a particular example we study ECG inversion in further detail. ECG inversion is a linear inverse source problem with boundary measurements governed by an anisotropic diffusion equation, which naturally cries for solutions under changing geometries, namely the beating heart. 

We formulate a regularization strategy using Tikhonov regularization and, using standard source conditions, we prove convergence rates. A special property of our approach is that not only operator perturbations are introduced by the diffuse domain method, but more important we have to deal with topologies which depend on a parameter $\eps$ in the diffuse domain method, i.e. we have to deal with $\eps$-dependent forward operators and $\eps$-dependent norms. In particular the appropriate function spaces for the unknown and the data depend on $\eps$. This prevents to apply some standard convergence techniques for inverse problems, in particular interpreting the perturbations as data errors in the original problem does not yield suitable results. We consequently develop a novel approach based on saddle-point problems.

The numerical solution of the problem is discussed as well and results for several computational experiments are reported. In particular investigations of convergence rates support our theoretical findings.
\end{abstract}

\maketitle



\section{Introduction}

Mathematical models based on differential and integral equations to be solved on complex or time-varying domains play an important role in many applications, in particular in biomedicine due to the complexity and inherent motion of living systems. A straight-forward approach towards the numerical solution of such problems is to resolve the geometries by building appropriate grids and subsequent computation on those e.g. via finite element or finite volume methods. Due to the high complexity of building grids and interpolation issues between different time steps several approaches have emerged that avoid the explicit resolution of the geometry and rather work on a fixed grid, either directly by adapting the discretization scheme (cf. \cite{BastianEngwer2009,HackbuschSauter1997,LiehrPreusserRumpfSauterSchwen2009}) or by implicitly representing the geometry in terms of characteristic functions, level set functions or diffuse interfaces (cf. \cite{bertalmio2003variational,burger2009finite,GlowinskiPanPeriaux1994,LeVequeLin1994,LiLowengrubRaetzVoigt2009,LervagLowengrub2014,ratz2006pde}). In the latter approach the interface is encoded via a function $\varphi^\eps$ that takes values close to $+1$ in the interior and $-1$ in the exterior of the domain to be represented, with an interfacial layer of smooth transition, which has a size of order $\eps$. This approach is highly motivated by Cahn-Hilliard and phase-field models in materials science (cf. \cite{alikakos1994convergence,cahn1958free,caginalp1989stefan}).

Analogous issues related to complex geometry frequently and increasingly arise in many inverse problems, e.g. in medical imaging shapes are obtained from segmentation of an anatomical imaging via MR or CT and subsequently used for other inversion tasks such as emission tomography or electromagnetic inversion (like EEG, MEG, ECG, MCG). Diffuse interface methods have however hardly been considered (cf. \cite{Styles}), and in particular their convergence analysis has not been worked out in relation to regularization methods, which introduce another small parameter. To be more precise consider canonical inverse problems of the form
\begin{equation}
	A(u) = f, 
\end{equation}
where $A:{\mathcal X} \rightarrow {\mathcal Y}$ is the forward operator between function spaces and $f$ are noisy data. Those are to be solved by variational regularization techniques, which consist in minimizing 
\begin{equation} \label{variationalregularization}
	J(u) = \Vert A(u)-f \Vert_{\mathcal Y}^q + \alpha \Vert u - u_*\Vert_{\mathcal X}^r,
\end{equation}
with $q, r \geq 1$ and $u_*$ being a prior for the variable $u$, potentially equal to zero. There are three potential dependencies on the domain $D$. The first as direct dependence of the operator upon $D$, e.g. via partial differential equations to be solved on $D$ in order to evaluate $A$. 
The diffuse interface method will introduce an approximation of the form
\begin{equation} \label{variationalregularizationeps}
	J^\eps(u) = \Vert A^\eps(u)-f^\eps \Vert_{\mathcal Y^\eps}^q + \alpha \Vert u - u_*\Vert_{\mathcal X^\eps}^r,
\end{equation}
with appropriate perturbations of operator, data, and norms. In particular the last fact creates novel theoretical questions, since the topologies of the $\eps$-dependent space might not be equivalent to the ones of the original spaces $\mathcal{X}$ and $\mathcal{Y}$ as we shall see below. The convergence analysis thus needs to go beyond the current state of the theory and in this paper we use a novel approach based on saddle-point formulations. We also mention that our analysis does not mainly target the case of $\eps \rightarrow 0$ for fixed $\alpha$, which could be derived with similar techniques as used here and in \cite{BES2014}.

We mention that from a practical point of view there are further reasons that can make diffuse interface methods attractive. A quite peculiar property is that due to the ill-posedness of most inverse problems and the consequently limited resolution of regularization methods high frequency information is lost. Intuitively this should also concern fine details in the geometry, hence smearing out the geometry information might not harm the quality of reconstructions or even further stabilize the problem. Another aspect is uncertainty in geometries, which may concern the domain (e.g. from incorrect segmentations) as well as the measurement locations (e.g. electrode positions on the body surface in EEG and ECG). A diffuse interface that averages the model over different possible domain shapes seems hence more appropriate than an exact treatment of the interface.
A detailed study of these aspects is left to future research.

In the construction of diffuse interface methods we follow the approach in \cite{BES2014}. 
During the whole paper we shall assume to have a representation of an unknown shape $D \subset \Omega$ via its signed distance function $d_D$, i.e.,
\begin{equation}
	d_D(x) = \left\{ \begin{array}{ll} + \text{ dist}(x,\partial D) & \text{if } x \in \Omega \setminus D, \\ - \text{ dist}(x,\partial D) & \text{if } x \in D. \end{array}\right.
\end{equation} The diffuse interface is then constructed via 
\begin{equation} \label{phiepsilon}
	\varphi^\eps = S\left(-\frac{d_D}{\eps}\right)  
\end{equation}
for $\eps > 0$ small and $S$ being a sigmoidal function, i.e., increasing with 
$	\lim_{t \rightarrow \pm \infty} S(t) =  \pm 1$.
As $\eps$ tends to zero, $S$ converges to the sign function, and hence $\varphi^\eps$ formally converges to 
\begin{equation} \label{phizero}
	\varphi^0(x) =  \left\{ \begin{array}{ll}  - 1 & \text{if } x \in \Omega \setminus D, \\ +1 & \text{if } x \in D. \end{array}\right.
\end{equation}
Indeed this convergence can easily be made rigorous in $L^p$-spaces. 
In this work we use the sigmoidal function $S:\RR\to\RR$ defined by $S(t)=t/|t|$ for $|t|\geq 1$ and $S(t)=t$ for $|t|<1$; more general choices are allowed and we refer the reader to \cite{BES2014}. 
Note that the support of $\nabla \varphi^\eps$ is restricted to an $\eps$-neighborhood of $\partial D$ and that $\varphi^\eps$ is a Lipschitz-continuous function bounded by $\pm 1$.


In order to obtain a representation with diffuse interfaces, we mainly need to discuss the approximation of integrals over the domain and its boundary. With such we can obviously treat most relevant issues: integral equations, inverse problems for partial differential equations via weak formulations, data fidelities and regularization terms in variational regularization methods. The only relevant case that needs additional considerations seems to be the different use of tangential and normal derivatives on curves or surfaces, which we postpone to future considerations. The key idea to approximate such integrals is a weighted averaging of the integrals on $\{ d_D < t \}$ instead of the original domain 
$\{ d_D < 0 \}$ only (and similar for boundary integrals). Since 
$\frac{1}{2\eps} S'(\frac{\cdot}\eps)$ approximates a concentrated distribution at zero, we expect
\begin{align*}
\int_D g(x)\d x &= \int_{\{ d_D < 0 \}} g(x)\d x 
= \int_{-\infty}^\infty \frac{1}{2\eps} S'(-\frac{t}\eps) \int_{\{ d_D < 0 \}} g(x)\d x\d t \\
&\approx \int_{-\infty}^\infty \frac{1}{2\eps} S'(-\frac{t}\eps) \int_{\{ d_D < t \}} g(x)\d x\d t \\
&= \frac{1}2 \int_{-1}^{1} \int_{\{\varphi^\eps > s\}} g(x)\d x \d s, 
\end{align*}
where we have used the substitution $s=S(-\frac{t}\eps)$ in the last term. Now the layer cake-representation can further be used for given integrable $g$  to rewrite 
$$ \int_{-1}^{1} \int_{\{\varphi^\eps > s\}} g(x)\d x\d t= \int_\Omega \int_{-1}^{\varphi^\eps(x)}  \d s g(x)\d x =\int_\Omega (1+\varphi^\eps)(x) g(x) \d x. $$

By an analogous computation we obtain for the boundary integral
$$ \int_{\partial D} g(x)\d \sigma(x) \approx  \frac{1}2 \int_{-1}^{1} \int_{\{\varphi^\eps = s\}} g(x)\d \sigma(x) \d s, $$
which can be simplified via the co-area formula to 
$$  \int_{-1}^1 \int_{\partial \{\varphi^\eps = s\}} g(x) \d\sigma(x) \d t =  \int_\Omega g(x) |\nabla \varphi^\eps(x)|\d x.$$
Detailed convergence results for these kind of integrals can be found in \cite{BES2014} and are recalled in the appendix.

Thus, integral functionals in \eqref{variationalregularization} of the form
\begin{equation}
	{\mathcal F}_{dom}(v) = \int_D \Psi(v,\nabla v,\ldots,\nabla^m v)\d x
\end{equation}
are approximated in a straight-forward way as
\begin{equation}
	{\mathcal F}_{dom}^{\eps}(v) = \int_\Omega \Psi(v,\nabla v,\ldots,\nabla^m v) (1+\varphi^\eps) \d x.
\end{equation}
Functionals on surfaces are less straight-forward with the exception of simple $L^p$-type regularization functional 
$$ {\mathcal F}_{bound}(v) = \int_{\partial D} \Psi(x,v) \d\sigma(x), $$
which have an obvious approximation
$$ {\mathcal F}_{bound}^\eps(v) = \int_\Omega \Psi(\cdot,v) |\nabla \varphi^\eps(x)| \d x. $$
Gradient or higher-order derivative based regularization on surfaces is usually formulated in terms of tangential derivatives, whose diffuse approximation solely based on $\varphi^\eps$ is more involved. In this paper we will however restrict our attention to $L^2$-norms on the boundary of a domain, which can be approximated as ${\mathcal F}_{bound}$ above. From the construction we see however that the diffuse version of an $L^2$-norm (defined as the square root of ${\mathcal F}_{bound}$ with square $\Psi$) has an important topological difference to the $L^2$-norm on the sharp interface. Note that the latter roughly corresponds to an $H^{1/2}$-norm on the domain via trace theorems, hence the diffuse norm induces a weaker topology. 

In the remainder of the paper we work out the convergence analysis of the diffuse interface approximation \eqref{variationalregularizationeps} in the example of ECG inversion, i.e. the solution of an elliptic Cauchy problem. This problem is well-studied on the one-hand from a theoretical point of view, but on the other hand leaves a clear practical challenge of efficient solution on different complex domains (moving hearts). More importantly, it includes a lot of the potential challenges for the convergence analysis: Both the unknown as well as the data are functions on parts of the boundary to be approximated by diffuse interfaces and the forward operator is also defined via a partial differential equation on the (diffuse) domain. We discuss the problem and its diffuse approximation in Section 2, before we proceed to the convergence analysis in Section 3. We show that the diffuse regularized solution converges to the correct solution as $\alpha$, $\eps$ and the noise level $\delta$ tend to zero under standard conditions on $\alpha$ and roughly for $\eps \sim \alpha$ (or some higher power of $\alpha$). In the case of correct solutions satisfying a standard source condition (cf. \cite{EHN96}) and a standard choice $\alpha \sim \delta$ we obtain an optimal convergence rate if $\eps \sim \delta^{2/3}$. This confirms our intuition that $\eps$ can be chosen rather large for inverse problems in presence of noise. Finally we discuss the numerical solution of the problem in Section 5 and provide a collection of experiments, whose results support our theory respectively indicate that one might obtain even better convergence rates with respect to $\eps$.

\section{Motivating Example: ECG Inversion}

In order to clarify the application of the diffuse domain method to the solution of an inverse problem, we study the following setup encountered in the reconstruction of epicardial 
potentials from ECG body surface potential measurements. Given data $f$, which are samples of 
the potential $v$ (more precisely its Dirichlet trace on the body surface $\partial B$) we want to reconstruct 
the epicardial potential, i.e., the trace of $v$ on $\partial H$, where $H \subset B$ is the heart volume. Here we use a so-called 
flux-based formulation, i.e., we use the Neumann boundary value $u$ on $\partial H$ as the 
unknown for the inversion, i.e., the forward model in weak form is
\begin{equation}
	\int_{D} M \nabla v \cdot \nabla w \d x = \int_{\partial H} u w \d\sigma 
	\quad \text{for all } w \in H_\diamond^1(D). \label{eq:sharpweakform}
\end{equation}
with $D=B\setminus \overline{H}$ and 
$$H^{1}_\diamond(D)=  \{ w\in H^{1}(D): \int_{\partial H} w \d\sigma = 0\}.$$
This formulation has been found to be quite appealing in the ECG-inversion problem, 
in particular when variational regularization is formulated on $u$ rather than the 
Dirichlet trace of $v$ (cf. \cite{ghosh2009application,khoury1994use,throne2000comparison}). The epicardial potential can be 
computed subsequently from the forward model. Note that \eqref{eq:sharpweakform} is 
the weak formulation of the anisotropic Laplace equation 
$\nabla \cdot ( M \nabla v) = 0$ with Neumann boundary conditions, with zero flux 
on $\partial B$. The latter is natural due to the insulation of the body. 


In the whole manuscript we will assume the following ellipticity condition:
There exists a constant $m>0$ such that
\begin{align}\label{eq:ellipticity}
 m |\xi|^2 \leq \xi\cdot M(x)\xi \leq \frac{1}{m} |\xi|^2\quad\text{for all } x,\xi\in \RR^n.
\end{align}
Moreover, we will always assume the following regularities: $\partial D \in C^{3,1}$, $M\in W^{2,\infty}(\Omega)$ and $v\in W^{3,\infty}(D)$ being the solution of \eqref{eq:sharpweakform}. Thus, $n\cdot M\nabla v \in W^{2,\infty}(\partial D)$. These regularity assumptions can be weakened at the cost of worse approximation properties of the diffuse domain method, see some remarks below and \cite{BES2014}.

\begin{lemma}\label{lem:existence}
 Let \eqref{eq:ellipticity} hold. Then, for any $u\in L^2(\partial H)$, there exists a unique $v\in H^1_\diamond(D)$ such that \eqref{eq:sharpweakform} holds. In particular, there exists a constant $C>0$ such that
 \begin{align*}
  \| v \|_{H^1(D)} \leq C \|u\|_{L^2(\partial H)}.
 \end{align*}
\end{lemma}
\begin{proof}
 Due to the Poincar\'e inequality the bilinear form on the left-hand side of \eqref{eq:sharpweakform} defines an inner product on $H^1_\diamond(D)$. For $u\in L^2(\partial H)$ the right-hand side of \eqref{eq:sharpweakform} defines a bounded linear functional on $H^1_\diamond(D)$. An application of the Lax-Milgram lemma yields the assertion.
\end{proof}




\subsection{Forward map and inverse problem.} 
We define a linear operator
\begin{equation}\label{eq:sharpIPdef}
	F:L^2(\partial H)\to L^2(\partial B), \quad Fu=v_{\mid \partial B}
\end{equation}
with $v\in H^1_\diamond(D)$ being the solution to \eqref{eq:sharpweakform} with $u\in L^2(\partial H)$. The inverse problem we are concerned with is the following. For given $f \in L^2(\partial B)$ determine $u\in L^2(\partial H)$ such that
\begin{align}\label{eq:sharpIP}
 F u = f\qquad \text{in } L^2(\partial B).
\end{align}
The following lemma collects some basic properties of the forward map $F$.
\begin{lemma}\label{lem:uniqueness}
 The forward map $F:L^2(\partial H)\to L^2(\partial B)$ defined by \eqref{eq:sharpIPdef} is linear, injective, bounded and compact. 
\end{lemma}
\begin{proof}
Linearity is obvious. Compactness, and hence boundedness, follows from compactness of the trace operator $H^1(D)\to L^2(\partial B)$ and Lemma~\ref{lem:existence}.
To show injectivity, let $u_1, u_2\in L^2(\partial H)$ such that $F u_1 = f= F u_2$, and denote by $v_1$, $v_2$ the corresponding solutions to \eqref{eq:sharpweakform}. 
Then the difference $w=v_1-v_2$ is a weak solution to the Cauchy problem
\begin{align*}
 -{\rm div}(M\nabla w)=0\quad\text{in }D,\qquad n\cdot M\nabla w=0  \text{ on } \partial B,\quad w=0 \text{ on }\partial B.
\end{align*}
Since $M$ is Lipschitz, the Cauchy problem is uniquely solvable \cite{Miranda70},
i.e., $w=0$ and $u_1=u_2$.
\end{proof}
In view of Lemma~\ref{lem:uniqueness} and since it is easy to see that the range of $F$ is infinite-dimensional, the inverse problem \eqref{eq:sharpIP} is ill-posed, and some sort of regularization is needed for a stable inversion of \eqref{eq:sharpIP}.
In the whole manuscript, we denote by $f^\dagger, v^\dagger$ and $u^\dagger$ the exact data and solutions respectively.

\subsection{Variational Regularization with Sharp Interfaces}
As basic regularization method we consider the following Tikhonov type functional
\begin{align}\label{eq:Tikhonov}
	J(u,v) = \frac{1}{2} \|v-f^\delta\|_{L^2(\partial B)}^2 + \frac{\alpha}{2} \| u\|^2_{L^2(\partial H)} \quad \text{subject to \eqref{eq:sharpweakform}},
\end{align}
where $f^\delta \in L^2(\partial B)$ represents noisy data for which we assume that
\begin{align}\label{eq:noise_level}
 \|f^\dagger - f^\delta \|_{L^2(\partial B)} \leq \delta.
\end{align}

As pointed out in the introduction, in applications we have in mind the sharp interfaces $\partial B$ and $\partial H$ are not known exactly, and we aim at employing the diffuse integrals introduced above. The quadratic case seems to be sufficient to understand the main difficulties arising from the diffuse approximation, extensions to other $L^p$-norms can be made with analogous arguments as in the sharp interface case.
%
%
%
%
%

\begin{rem}\label{rem:convergence}
Considering the reduced functional $\hat J(u) = J(u,F(u))$, which is quadratic and strictly convex, we obtain from \cite[Thm~5.2]{EHN96} that the minimizers $u_{\alpha,\delta}$ of $\hat J$ with $f^\dagger$ replaced by $f^\delta$ converge to $u^\dagger$ in $L^2(\partial H)$ as long as $u^\dagger \in L^2(\partial H)$, $\|f^\dagger-f^\delta\|_{L^2(\partial B)}\leq \delta$ and $\alpha\to 0$ is chosen such that $\delta^2/\alpha\to 0$ as $\delta\to 0$, i.e.,
 $\lim_{\delta\to 0} u_{\alpha,\delta} = u^\dagger$.
\end{rem}

\subsection{Variational Regularization with Diffuse Interface}

In the following we discuss a diffuse approximation of the variational problems introduced above. In order to distinguish the two different parts of the boundary $\partial D=\partial B \cup \partial H$ we choose a weight $\gamma_H$ that equals one in a neighborhood of $\partial H$
and zero in a neighborhood of $\partial B$. Vice versa, we choose a second weight $\gamma_B$, which equals one in a neighborhood of the measurement locations on $\partial B$
and vanishes in a neighborhood of $\partial H$.

\subsubsection{Sobolev Spaces}
To define a suitable function space, let us introduce the scalar product
\begin{align*}
	\langle v, w \rangle_{\H^\eps} 
	=\langle \nabla v, \nabla w \rangle_{\omega^\eps} +\langle v, w \rangle_{\omega^\eps} =\int_{\Omega} ( \nabla v \cdot \nabla w + v w)\omega^\eps \d x,
\end{align*}
where $\omega^\eps=(1+\varphi^\eps)/2$, and the corresponding weighted Sobolev space defined by
\begin{align*}%
	\H^\eps := \{v \in L^2({\Omega})| \|v\|_{\H^\eps}^2=\langle v,v\rangle_{\H^\eps} < \infty\}.
\end{align*}
Note that we tacitly identify functions $v$ and $w$ if $v=w$ on ${\rm supp}(\omega^\eps)$ in order to make $\|\cdot\|_{\H^\eps}$ a norm. 
Moreover, we denote by $L^p(\omega^\eps)=L^p(\Omega;\omega^\eps)$ and $W^{k,p}(\omega^\eps)=W^{k,p}(\Omega,\omega^\eps)$ the corresponding weighted Lebesgue and Sobolev spaces; see \cite{BES2014} for details. In particular $\H^\eps=W^{1,2}(\omega^\eps)$. We will also write $L^p(\tilde \Omega;\gamma)$ with some weighting function $\gamma$ and $\tilde \Omega \subset \Omega$ to denote the corresponding weighted Lebesgue space.
One observes that due to the properties of $\omega^\eps$, we have $\sqrt{2} \|v\|_{\H^\eps} \geq \|v\|_{H^1(D)}$.
Thus, any uniform estimate and convergence in the norm of $\H^\eps$ can be transfered immediately to the norm of $v$ in $H^1(D)$, which is a relevant quantity to understand the approximation properties; for further details on the spaces $\H^\eps$ see also \cite{BES2014}.
%
For the interface variable $u$ we consider the space $\Ue=L^2(\gamma_H|\nabla\omega^\eps|)$ with corresponding inner product $\langle\cdot,\cdot\rangle_{\Ue}$; and for the measurements $f$ we use $\Me=L^2(\gamma_B|\nabla\omega^\eps|)$ with corresponding inner product $\langle\cdot,\cdot\rangle_{\Me}$; i.e. for $u,q\in\Ue$ and $f,v\in\Me$
$$
 \langle u,q\rangle_{\Ue} = \int_{\Omega} uq |\nabla \omega^\eps|\gamma_H\d x,\qquad\langle v,f\rangle_{\Me} = \int_{\Omega} vf |\nabla \omega^\eps|\gamma_B\d x.
$$
As above, we identify functions $u,q\in\Ue$ if $u=q$ on ${\rm supp}(|\nabla\omega^\eps|\gamma_H)$.
The diffuse trace lemma~\ref{lem:diffusetrace} shows that the embedding $\H^\eps\hookrightarrow \Ue$ is continuous.
For appropriate normalization, we will also consider space 
\begin{align}\label{eq:hspace}
 \H^\eps_{\diamond}=\{v\in \H^\eps: \langle v,1\rangle_{\Ue} =0\}. 
\end{align}
%
As $\partial D$ is smooth, there exists a continuous extension $E_{D,\Omega}:H^1(D)\to H^1(\Omega)$ \cite{Adams1975}, and we will write $v$ instead of $E_{D,\Omega}v$ to evaluate functions in $H^1(D)$ in $\Omega$.

\subsubsection{Extensions constant off the interface}\label{sec:extensions}

We consider extensions constant off the interfaces $\partial H$ and $\partial B$, respectively. 
Note that for $0<\eps \leq \eps_0$, with $\eps_0$ sufficiently small, which we will assume throughout the paper, and for each $x\in {\rm supp}(|\nabla\omega^\eps|)$ there exists a unique $\bar x\in\partial D$ such that $x=\bar x + d_D(x) n(\bar x)$; see \cite{DelfourZolesio2011}.
We can then define $E_H:L^2(\partial H)\to \Ue$ by
\begin{align*}
 E_Hu(x) = \tilde u(x) = u(\bar x),\qquad  x = \bar x + d_D( x) n(\bar x)\in {\rm supp}(\gamma_H |\nabla\omega^\eps|), \ \bar x \in\partial H,
\end{align*}
and similarly for the measurements, $E_B:L^2(\partial B)\to \Me$ given by
\begin{align*}
 E_B f(x) = \tilde f(x) = f(\bar x),\qquad  x = \bar x + d_D( x) n(\bar x)\in {\rm supp}(\gamma_B |\nabla\omega^\eps|), \ \bar x \in\partial B.
\end{align*}
If the context is clear, we will write in abuse of notation $\tilde u$ and $\tilde f$ instead of $E_H u$ and $E_B f$.
Some properties of the extensions $E_B$ and $E_H$ are compiled in the appendix.

\subsubsection{Diffuse forward operator}\label{sec:diffuse_forward}
We approximate \eqref{eq:sharpweakform} via 
\begin{align}\label{eq:diffuseweakform}
	\langle M \nabla v,\nabla w\rangle_{\omega^\eps} = \langle u, w \rangle_{\Ue} \quad\text{for all } w \in \H_\diamond^\eps,
\end{align}
where $u\in \Ue$. We have the following well-posedness result for \eqref{eq:diffuseweakform}; see \cite[Lemma~6.17]{BES2014}.

\begin{lemma}\label{lem:existence_diffuse}
 For each $u\in \Ue$ there exist a unique $v\in\H^\eps_\diamond$ verifying $\eqref{eq:diffuseweakform}$ and a constant $C>0$ independent of $\eps$ such that
 \begin{align*}
  \| v\|_{\H^\eps} \leq C \|u\|_{\Ue}.
 \end{align*}
\end{lemma}
In order to use $u^\dagger$ in \eqref{eq:diffuseweakform}, we will use the extension $\tilde u^\dagger = E_H u^\dagger\in\Ue$. This will introduce errors quantified by the following 
\begin{lemma}\label{lem:perturbation}
 Let $v^\eps\in \H^\eps_\diamond$ be a solution to \eqref{eq:diffuseweakform} with data $u$ replaced by $E_H u^\dagger$. Then there exists $C>0$ such that
 \begin{align*}
  \|v^\dagger-v^\eps\|_{\H^\eps}\leq C \eps^{3/2} \|v^\dagger\|_{W^{3,\infty}(D)}.
 \end{align*}
\end{lemma}
\begin{proof}
 The difference $v^\eps-v^\dagger$ satisfies
 \begin{align*}
  \langle M\nabla (v^\eps-v^\dagger), \nabla w \rangle_{\omega^\eps} &= \langle \tilde u^\dagger, w\rangle_{\Ue}-\langle  M\nabla v^\dagger, \nabla w\rangle_{\omega^\eps}.
 \end{align*}
 Integration by parts and $-n|\nabla \omega^\eps| =\nabla\omega^\eps$ yields 
\begin{align*}
  \langle \tilde u^\dagger, w\rangle_{\Ue}-\langle M\nabla v^\dagger, \nabla w\rangle_{\omega^\eps}
  =\langle \tilde u^\dagger-n\cdot M\nabla v^\dagger , w\rangle_{\Ue}
  - \langle n\cdot M\nabla v^\dagger,  w \rangle_{\Me}-\langle {\rm div}(M\nabla v^\dagger), w\rangle_{\omega_\eps}.
\end{align*}
To treat the first term we use $n\cdot M\nabla v^\dagger = u^\dagger$ on $\partial H$ and Lemma~\ref{lem:error_diffuse_integral} (iv) to obtain
\begin{align*}
 \langle E_H (n\cdot M\nabla v^\dagger)-n\cdot M\nabla v^\dagger, w\rangle_{\Ue} \leq C \eps^{3/2} \| v^\dagger\|_{W^{3,2}(\Omega)} \|w\|_{\H^\eps}
\end{align*}
for some $C>0$.
Since $n\cdot M\nabla v^\dagger=0$ on $\partial B$, the second term can be treated similarly. 
To treat the third term we use ${\rm div}(M\nabla v^\dagger)=0$ in $D$ and Lemma~\ref{lem:error_diffuse_integral} (i) to obtain
\begin{align*}
 |\langle {\rm div}(M\nabla v^\dagger),w\rangle_{\omega^\eps}| \leq C\eps^{3/2} \|{\rm div}(M\nabla v^\dagger)\|_{W^{1,\infty}(\Omega)} \|w\|_{\H^\eps}.
\end{align*}
The a priori estimate of Lemma~\ref{lem:existence_diffuse} yields the assertion. 
\end{proof}


Since in applications we have in mind both $\partial B$ and $\partial H$ are unknown or difficult to approximate, we will employ diffuse approximations of $\partial B$ and $\partial H$. Hence, we are concerned with solving the following (diffuse) operator equation
\begin{align}\label{eq:diffuseIP}
 F^\eps u = \tilde f^\delta \quad \text{in } \Me,
\end{align}
where $F^\eps: \Ue\to\Me$ is a bounded linear operator mapping $u$ onto the diffuse trace of the solution $v$ of \eqref{eq:diffuseweakform}. The data $\tilde f^\delta=E_B f^\delta$ is obtained by extending the measured data $f^\delta$.
In view of the possible extensions of the interface data $u$ and $f$, there are of course many different possibilities to define a forward operator. Since these investigations will be similar to ours, we leave the modifications to the reader.
Notice that, for each $\eps>0$ fixed, the injection $\H^\eps \hookrightarrow \Me$ is compact, and hence \eqref{eq:diffuseIP} is ill-posed as well.

As $E_B$ is bounded, see Lemma~\ref{lem:extension}, measuring in the weaker diffuse interface norm will not alter the noise level significantly, i.e.,
\begin{align}\label{eq:noise_extended}
 \| E_B f^\dagger- E_B f^\delta\|_{\Me}\leq C(\eps) \delta,
\end{align}
with $C(\eps)\to 1$ as $\eps\to 0$.
Using the diffuse domain method as underlying governing equation will however have an impact, which might be interpreted as an operator perturbation, namely
\begin{align*}
 \| F^\eps E_H u^\dagger - E_B f^\delta\|_{\Me}\leq C (\delta + \eps^{3/2}).
\end{align*}
The latter estimate is a direct consequence of the triangle inequality and Lemma~\ref{lem:perturbation}.
%
The Tikhonov functional \eqref{eq:Tikhonov} is approximated by the following functional
\begin{align}\label{eq:Tikhonov_diffuse}
  J^\eps(u,v)= \frac{1}{2} \| v - \tilde f^\delta \|^2_{\Me} + \frac{\alpha}{2} \|u\|^2_{\Ue}\quad\text{subject to } \eqref{eq:diffuseweakform}.
\end{align}
Note that we not only have to deal with perturbed forward operators but also with perturbed data misfit and regularization functionals. As the diffuse boundary norms are weaker than their counterparts for the sharp interfaces, this choice of topologies makes our investigations non-standard and requires adapted arguments to be detailed in the next section.

\section{Analysis of the Diffuse Domain Regularization}\label{sec:analysis}
In the following we provide an analysis of the variational models with diffuse interfaces.
We begin with the existence of minimizers of \eqref{eq:Tikhonov_diffuse} by investigating the associated saddle-point problem. Then we show stability and convergence of minimizers of the diffuse Tikhonov functional. Under a standard source condition we then also obtain convergence rates.

\subsection{Saddle-Point Formulation}\label{sec:saddle}
In the following we consider variations of the Lagrangian corresponding to \eqref{eq:Tikhonov_diffuse}
\begin{align}\label{eq:Lagrangian_diffuse}
	L^\eps(u,v,p) = J^\eps(u,v) + \langle M \nabla v, \nabla p\rangle_{\omega^\eps} -\langle u, p\rangle_{\Ue}.
\end{align}
Therefore, let us define two bilinear forms, namely
$a^\eps: (\Ue \times \H^\eps_\diamond)\times (\Ue \times \H^\eps_\diamond) \rightarrow \R$ given by
\begin{align*}
    a^\eps(u,v;q,w) = \langle v, w\rangle_{\Me} + \alpha \langle u,q\rangle_{\Ue},
\end{align*}
and $b^\eps: (\Ue\times \H^\eps_\diamond) \times \H^\eps_\diamond \rightarrow \R$ given by
\begin{align*}
	b^\eps(u,v;p) = \langle M \nabla v, \nabla p\rangle_{\omega^\eps} -\langle u, p\rangle_{\Ue}.
\end{align*}
Saddle-points of $L^\eps$ are then characterized as solutions of 
\begin{equation} \label{saddlepoint}
\begin{array}{lclcll}
  a^\eps(u,v;q,w) &+ &b^\eps(q,w;p)&= & f^\eps(q,w) & \quad \text{for all }(q,w) \in  \Ue \times \H^\eps_\diamond,\\
  b^\eps(u,v;r) & &    &= & g^\eps(r) & \quad \text{for all } r \in \H^\eps_\diamond.
\end{array}
\end{equation}
Here, we use the linear functionals $g^\eps: \H^\eps_\diamond \rightarrow \RR$, $g^\eps(r)=0$, and $f^\eps :\Ue \times \H^\eps_\diamond \rightarrow \RR$, $f^\eps(q,w)=\langle \tilde f^\delta,w\rangle_{\Me}$.
%
For the analysis of the saddle-point problem, let us define
\begin{align*}
 \|(u,v)\|_\alpha^2 = \alpha (\|u\|_{\Ue}^2 + \|\nabla v\|_{L^2(\omega^\eps)}^2) + \|v\|_{\Me}^2,
\end{align*}
which is a norm equivalent to the natural norm on $\Ue\times \H^\eps_\diamond$ for fixed $\alpha>0$; cf. Lemma~\ref{lemma:diffusefriedrichs}.

Let us first collect some basic properties of the saddle-point problem and the associated bilinear forms:
\begin{lemma}[Continuity]\label{lem:continuity}
 Let $0<\alpha\leq \alpha_0$.
 Then there exists a constant $C_c$ independent of $\eps$ and $\alpha$ such that
 \begin{align*}
  |a^\eps(u,v;q,w)|  \leq C_c \|(u,v)\|_\alpha \|(q,w)\|_\alpha\quad\text{and}\quad 
  |b^\eps(u,v;p)| \leq  \frac{1}{\sqrt{\alpha}}C_c \|(u,v)\|_{\alpha} \|p\|_{\H^\eps}
 \end{align*}
 for all $(u,v),(q,w)\in \Ue\times\H^\eps_\diamond$ and $p\in\H^\eps_\diamond$.
\begin{proof}
 The estimates follow from Lemma \ref{lem:diffusetrace} and a standard Cauchy-Schwarz argument.
\end{proof}
\end{lemma}

\begin{lemma}[Kernel ellipticity]

Let $0<\alpha\leq \alpha_0$.
Then there exists a constant $C_e$ independent of $\eps$ and $\alpha$ such that
\begin{equation}
	a^\eps(u,v;u,v) \geq C_e \|(u,v)\|^2_{\alpha}
\end{equation}
for all $(u,v)\in \Ue\times\H^\eps_\diamond$ such that $b^\eps(u,v;v) = 0$.
\end{lemma}
\begin{proof}
Using $b^\eps(u,v;v)=0$ we obtain for any $\kappa > 0$ 
\begin{align*}
a^\eps(u,v;u,v)  &=  a^\eps(u,v;u,v) + \kappa b^\eps(u,v;v)\\
&\geq  \|v\|_{\Me}^2 +\alpha \|u\|^2_{\Ue} + \kappa m\|\nabla v\|_{L^2(\omega^\eps)}^2 -\kappa\|u\|_{\Ue}\|v\|_{\Ue}\\
&\geq  \|v\|_{\Me}^2 + \frac{\alpha}{2} \|u\|^2_{\Ue} + \kappa m\|\nabla v\|_{L^2(\omega^\eps)}^2 -\frac{\kappa^2}{2\alpha}\|v\|_{\Ue}^2,
\end{align*}
where we have used \eqref{eq:ellipticity} and Young's inequality.
With Lemma~\ref{lem:diffusetrace} and Lemma~\ref{lemma:diffusefriedrichs} there exists a constant $c>0$ independent of $\eps$ such that
\begin{align*}
 \|v\|_{\Ue}^2 \leq c (\|\nabla v\|_{L^2(\omega^\eps)}^2 + \|v\|_{\Me}^2).
\end{align*}
Increasing $c$ if necessary, we may assume that $c\geq \alpha_0 m^2$.
Hence, we arrive at the estimate
\begin{align*}
 a^\eps(u,v;u,v) \geq \|v\|_{\Me}^2 + \frac{\alpha}{2} \|u\|^2_{\Ue} + \kappa m\|\nabla v\|_{L^2(\omega^\eps)}^2 -\frac{\kappa^2 c}{2\alpha} (\|\nabla v\|_{L^2(\omega^\eps)}^2 + \|v\|_{\Me}^2).
\end{align*}
Choosing $\kappa=m\alpha/c$ we have that
\begin{align*}
 a^\eps(u,v;u,v) \geq (1-\frac{m^2\alpha}{2c})\|v\|_{\Me}^2 + \frac{\alpha m^2}{2c} \big(\|u\|^2_{\Ue} + \|\nabla v\|_{L^2(\omega^\eps)}^2\big).
\end{align*}
By choice of $c$, $1-\frac{m^2\alpha}{2c}\geq \frac{1}{2}$, and the assertion holds with $C_e=\min\{1,m^2/c\}/2$.
\end{proof}
\begin{lemma}[Inf-sup stability]\label{lemma:infsup}
Let $0<\alpha\leq \alpha_0$.
Then there exists a constant $C_i$ independent of $\eps$ and $\alpha$ such that
\begin{equation}
	\sup_{(u,v) \in \Ue \times \H^\eps_\diamond} \frac{b^\eps(u,v;p)}{ \|(u,v)\|_\alpha} \geq C_i \| p \|_{\H^\eps}\quad \text{for all } p \in \H^\eps_\diamond.
\end{equation}
\end{lemma}
\begin{proof}
Let $p\in\H^\eps_\diamond$ be given. By Lemma~\ref{lem:diffusetrace} the embedding $\H^\eps_\diamond\hookrightarrow \Ue$ is continuous, and thus we can choose $v=p$ and $u=-p$.
Using Lemma~\ref{lemma:diffusefriedrichs} we further obtain another constant $C>0$, which possibly depends on $\alpha_0$ but not on $\eps$ or $\alpha$, such that $\|(u,v)\|_\alpha\leq C \|p\|_{\H^\eps}$.
The assertion then follows from
\begin{align*}
 b^\eps(u,v;p) \geq m \|\nabla p\|_{L^2(\omega^\eps)}^2 +\|p\|_{\Ue}^2
 \geq c \|p\|_{\H^\eps}^2,
\end{align*}
where we also applied \eqref{eq:ellipticity} and Lemma~\ref{lemma:diffusefriedrichs} with $\gamma=\gamma_H$.
\end{proof}
As a consequence of Brezzi's splitting theorem \cite{Brezzi74}, we obtain the following result. Note that the a priori estimates derived in \cite{Brezzi74} do not use the continuity constant of $b^\eps$.
\begin{thm}[Existence of saddle-points]\label{uniqsoln}
Let $0<\alpha\leq \alpha_0$. Then for each $f^\eps \in (\Ue \times \H^\eps_\diamond)'$ and $g^\eps \in (\H^\eps_\diamond)'$ there exist a unique solution
$ (u^\eps,v^\eps)\in \Ue \times \H^\eps_\diamond$ and $p^\eps\in\H^\eps_\diamond$ of \eqref{saddlepoint} and there exists a constant $C_E$ independent of $\eps$ and $\alpha$ such that 
\begin{align*}
  \alpha (\|u^\eps\|_{\Ue}^2 + \| \nabla v^\eps\|_{L^2(\omega^\eps)}^2) + \|v^\eps\|_{\Me}^2 + \|p^\eps\|_{\H^\eps}^2
     \leq C_E( \|f^\eps\|^2_{(\Ue \times \H^\eps_\diamond)'} + \|g^\eps\|_{(\H^\eps_\diamond)'}^2).
\end{align*}
\end{thm}
As usual $(\Ue \times \H^\eps_\diamond)'$ and $(\H^\eps_\diamond)'$ denote the respective dual spaces of $\Ue \times \H^\eps_\diamond$ and $\H^\eps_\diamond$, which we endow with the norms
\begin{align*}
 \|f^\eps\|_{(\Ue \times \H^\eps_\diamond)'}=\sup_{(u,v)\in \Ue\times \H^\eps_\diamond\setminus\{0\}} \frac{f^\eps(u,v)}{\|(u,v)\|_\alpha},\qquad \|g^\eps\|_{(\H^\eps_\diamond)'} =\sup_{p\in \H^\eps_\diamond\setminus\{0\}} \frac{g^\eps(p)}{\|p\|_{\H^\eps}}.
\end{align*}
\subsection{Convergence and Regularization properties}
In this section we will investigate the regularization properties of the diffuse domain method  when used in combination with Tikhonov regularization in more detail.

\begin{thm}[Stability]
 Let $f_1,f_2\in\Me$. Then, for $C_E$ from Theorem~\ref{uniqsoln}, we have that
 \begin{align*}
  \| (u_1^\eps-u_2^\eps, v^\eps_1-v_2^\eps)\|_\alpha \leq \sqrt{C_E} \|f_1-f_2\|_{\Me},
 \end{align*}
 where $(u_i^\eps,v_i^\eps)$, $i=1,2$, denotes the solution to \eqref{saddlepoint} with right-hand side $g^\eps=0$ and $f^\eps(q,w)=\langle f_i, w \rangle_{\Me}$.
\end{thm}
\begin{proof}
 $(u_1^\eps-u_2^\eps, v^\eps_1-v_2^\eps)$ is a solution to \eqref{saddlepoint} with right-hand side $g^\eps=0$ and $f^\eps(q,w)=\langle f_1-f_2, w \rangle_{\Me}$. Since $\|f^\eps\|_{(\Ue\times\H^\eps_\diamond)'} \leq \|f_1-f_2\|_{\Me}$ the result follows directly from Theorem~\ref{uniqsoln}.
\end{proof}

In order show convergence of the minimizers of the diffuse Tikhonov functional as $\alpha\to 0$, we need the following technical statement, which gives some sort of compactness.
\begin{prop}\label{prop:convergence1}
 Let $\{(u^\eps,v^\eps)\}\subset \Ue\times\H^\eps_\diamond$ be a sequence such that $b^\eps(u^\eps,v^\eps;r)=0$ for all $r\in \H^\eps_\diamond$ and such that there exists a constant $C>0$ with $\|u^\eps\|_{\Ue}\leq C$.
 Then there exists a subsequence $\{v^{\eps_k}\}$ of $\{v^\eps\}$ and $v\in H^1(\Omega)$ such that
 \begin{align*}
  \lim_{k\to \infty} \|\sqrt{\omega^{\eps_k}} \nabla v^{\eps_k}- \chi_D \nabla v\|_{L^2(\Omega)}=0 \quad\text{and}\quad \lim_{k\to \infty} \|v^{\eps_k}-v\|_{H^1(D)}=0.
 \end{align*}
\end{prop}
Here,  $\chi_D$ denotes the indicator function of $D$.

\begin{proof}
 Using Lemma~\ref{lem:existence_diffuse}, we obtain $\|v^{\eps}\|_{H^1(D)}\leq 2 \|v^\eps\|_{\H^\eps}\leq C\|u^\eps\|_{\Ue}\leq C$. Thus, we can extract a subsequence $\{v^{\eps}\}$, relabeled if necessary, such that $v^{\eps} \rightharpoonup v$ in $H^1(D)$ as $\eps\to 0$ for some  $v\in H^1(D)$.
 Now let $\varphi\in L^2(\Omega)^n$ be arbitrary. Since $0\leq \omega^\eps \leq 1$, we obtain $|\varphi \sqrt{\omega^\eps}| \leq |\varphi|\in L^2(\Omega)$. Moreover, since $\sqrt{\omega^\eps}\to \chi_D$ a.e. in $\Omega$ as $\eps\to 0$, we have $\varphi \sqrt{\omega^\eps} \to \varphi \chi_D$ a.e. in $\Omega$ as $\eps \to 0$. 
 Hence, using dominated convergence, $\varphi \sqrt{\omega^\eps}  \to \varphi \chi_D$ in $L^2(\Omega)^n$, and
\begin{align*}
 \int_D \sqrt{\omega^\eps} L\nabla v^\eps \cdot \varphi  \d x \to \int_D L\nabla v\cdot \varphi \d x\quad\text{as }\eps\to 0,
\end{align*}
using the Cholesky factorization $M=L^\top L$.
Since $\|\sqrt{\omega^\eps} L\nabla v^\eps\|_{L^2(\Omega)}\leq C \|v^\eps\|_{\H^\eps}$ is bounded (uniformly in $\eps$), and $|(\Omega\setminus D) \cap {\rm supp}(\omega^\eps)|\to 0$ as $\eps\to 0$, absolute continuity of the integral implies 
\begin{align*}
 \int_{\Omega\setminus D} \sqrt{\omega^\eps} L\nabla v^\eps \cdot \varphi \d x \leq \|\sqrt{\omega^\eps} L\nabla v^\eps\|_{L^2(\Omega)} \|\varphi\|_{L^2((\Omega\setminus D)\cap {\rm supp}(\omega^\eps))} \to 0\quad\text{as } \eps\to 0,
\end{align*}
i.e., $\sqrt{\omega^\eps} L\nabla v^\eps \rightharpoonup \chi_D L \nabla v$ in $L^2(\Omega)^n$ as $\eps \to 0$. 
 It remains to show that 
 $\|\sqrt{\omega^\eps} L\nabla v^\eps\|_{L^2(\Omega)}\to \|\chi_D L\nabla v\|_{L^2(\Omega)}$ as $\eps\to 0$.
 Testing $b^\eps(u^\eps,v^\eps,r)=0$ with $r=v^\eps-v-\langle v^\eps-v,1\rangle_{\Ue}/\langle 1,1\rangle_{\Ue}\in \H^\eps_\diamond$, and applying Cauchy-Schwarz's and Young's inequality yields
 \begin{align*}
  \|\sqrt{\omega^\eps} L\nabla v^\eps\|^2_{L^2(\Omega)}&=\langle M\nabla v^\eps,\nabla v\rangle_{\omega^\eps} + \langle r,u^\eps\rangle_{\Ue}\\
  &\leq \frac{1}{2}\|\sqrt{\omega^\eps}L\nabla v^\eps\|^2_{L^2(\Omega)} +\frac{1}{2}\|\sqrt{\omega^\eps}L \nabla v\|^2_{L^2(\Omega)} + \|r\|_{\Ue} \|u^\eps\|_{\Ue}.
 \end{align*}
 Since $\|r\|_{\Ue}\leq 2\|v^\eps-v\|_{\Ue}$, this reads as
 \begin{align}\label{eq:help1}
  \|\sqrt{\omega^\eps} L\nabla v^\eps\|^2_{L^2(\Omega)} \leq \|\sqrt{\omega^\eps}L \nabla v\|^2_{L^2(\Omega)} + 4 \|v^\eps-v\|_{\Ue} \|u^\eps\|_{\Ue}.
 \end{align}
  First, we observe by using Lebesgue's dominated convergence theorem that
  \begin{align*}
   \|\sqrt{\omega^\eps}L \nabla v\|^2_{L^2(\Omega)} \to  \int_D M\nabla v\cdot \nabla v\d x = \| \chi_D L\nabla v\|_{L^2(\Omega)}^2\quad\text{as } \eps \to 0.
  \end{align*}
  Next, we will show that $\|v^\eps-v\|_{\Ue}$ vanishes as $\eps \to 0$. By compactness of the embedding $H^1(D)\hookrightarrow L^2(\partial H)$, $v^\eps-v \rightharpoonup 0$ in $H^1(D)$ implies $v^\eps - v\to 0$ in $L^2(\partial H)$ by extracting another subsequence if necessary. Applying Theorem~\ref{thm:estimate} (i) to $v^\eps-v$ we obtain 
  \begin{align*}
   \|v^\eps-v\|_{\Ue} \leq C \sqrt{\eps}\|v^\eps-v\|_{\H^\eps} + \|v^\eps-v\|_{L^2(\partial H)} \to 0\quad \text{ as } \eps\to 0.
  \end{align*}
 By assumption $\{u^\eps\}$ is bounded in $\Ue$, and hence it follows from \eqref{eq:help1} that
 \begin{align}\label{eq:inter1}
  \limsup_{\eps\to 0} \|\sqrt{\omega^\eps}L \nabla v^\eps\|^2_{L^2(\Omega)}\leq \|\chi_D L\nabla v\|_{L^2(\Omega)}^2.
 \end{align}
Weak lower semicontinuity of the norm further implies 
\begin{align*}
  \|\chi_D L\nabla v\|_{L^2(\Omega)}\leq \liminf_{\eps\to 0}  \|\sqrt{\omega^\eps}L\nabla v^\eps\|_{L^2(\Omega)},
\end{align*}
i.e., $\|\sqrt{\omega^\eps} L\nabla v^\eps\|_{L^2(\Omega)}\to \|\chi_D L\nabla v\|_{L^2(\Omega)}$ as $\eps\to 0$, which yields the first assertion together with the ellipticity of $M$ (and consequent uniform bounds on the eigenvalues of $L$).

To show the second assertion, we infer from the first assertion that there exists another subsequence $\{\omega^\eps \nabla v^\eps\}$ which converges to $\chi_D \nabla v$ a.e. in $\Omega$ as $\eps\to 0$. As $\omega^\eps\geq 1/2$ on $D$ we further have that for this subsequence
$\nabla v^\eps$ converges to $\nabla v$ a.e. in $D$.
Moreover, with the same argument
$|\nabla v^\eps|^2 \leq 2 \omega^\eps |\nabla v^\eps|^2$ on $D$.
As $\omega^\eps |\nabla v^\eps|^2$ converges to $|\nabla v|^2$ in $L^1(D)$ by the first part, we obtain $\nabla v^\eps\to \nabla v$ in $L^2(D)$ by dominated convergence.
Compactness of the embedding $H^1(D)\hookrightarrow L^2(D)$ further yields $v^\eps\to v$ in $L^2(D)$ (for a subsequence), which concludes the proof.
\end{proof}

The next lemma basically resembles the a priori estimates of \cite{Brezzi74}. We state it explicitly since the structure of the estimate will be of importance below.
The proof is omitted.
%
\begin{lemma}\label{lem:apriori_Tik_eps}
 Let $(u_{\alpha,\delta}^\eps,v_{\alpha,\delta}^\eps,p_{\alpha,\delta}^\eps)$ be a saddle-point of $L^\eps$. Then there exists $C>0$ such that
 \begin{align*}
  \| v^\eps_{\alpha,\delta}-\tilde f^\delta\|_{\Me}^2 + \alpha \|u^\eps_{\alpha,\delta}\|^2_{\Ue} \leq C(\delta^2 + \alpha\| u^\dagger\|_{L^2(\partial H)}^2 + \eps^3 \|v^\dagger\|_{W^{3,\infty}(D)}^2).
 \end{align*}
\end{lemma}

Using similar assumptions as in the standard inverse problem theory \cite{EHN96}, we obtain the following convergence result.
\begin{thm}[Convergence]
 Let $\{(u_{\alpha,\delta}^\eps,v_{\alpha,\delta}^\eps,p_{\alpha,\delta}^\eps)\}$ be a sequence of saddle-points of $L^\eps$ for $\eps,\alpha,\delta>0$. If $\alpha$ and $\eps$ are chosen such that $\eps(\alpha,\delta)\to 0$ and $\alpha(\delta)\to 0$ as $\delta\to 0$, and $\delta^2/\alpha$ and $\eps^3/\alpha$ are bounded. Then there exists a constant $C$ independent of $\eps$, $\delta$ and $\alpha$ such that
 \begin{align*}
  \lim_{\delta\to 0} \|u^\eps_{\alpha,\delta} - \tilde u^\dagger\|_{(\H^\eps_\diamond)'}=0,\quad\text{and}\quad \|v^\eps_{\alpha,\delta}- \tilde f^\dagger\|_{\Me}\leq C \sqrt{\alpha}\quad\text{and}\quad \|v^\eps_{\alpha,\delta}- f^\dagger\|_{L^2(\partial B)}\leq C \sqrt{\alpha+\eps}.
 \end{align*}
\end{thm}

\begin{proof}
Applying \eqref{eq:noise_extended} and Lemma~\ref{lem:apriori_Tik_eps} yields
 \begin{align}\label{eq:conv_meas}
  \| v^\eps_{\alpha,\delta} -\tilde f^\dagger\|_{\Me} \leq \| v^\eps_{\alpha,\delta} -\tilde f^\delta\|_{\Me} + \|\tilde f^\delta-\tilde f^\dagger\|_{\Me}
  \leq C \sqrt{\alpha}
\end{align}
by choice of $\alpha$ and $\eps$.
The a priori estimate of Lemma~\ref{lem:apriori_Tik_eps} further asserts that
\begin{align*}
 \|u^\eps_{\alpha,\delta}\|_{\Ue}^2 \leq C(\frac{\delta^2}{\alpha} + \| u^\dagger\|_{L^2(\partial H)}^2 + \frac{\eps^3}{\alpha} \|v^\dagger\|^2_{W^{3,\infty}(D)}).
\end{align*}
Since $\delta^2/\alpha$ and $\eps^3/\alpha$ are bounded, $\|u^\eps_{\alpha,\delta}\|_{\Ue}$ is bounded (uniformly in $\eps$).
By Lemma~\ref{prop:convergence1} there exists $v\in H^1(D)$ such that for a subsequence, relabeled if necessary, $v^\eps_{\alpha,\delta} \to v$ in $H^1(D)$ as $\delta \to 0$.
Moreover, applying \eqref{eq:conv_meas} and Lemma~\ref{lem:error_diffuse_integral} (ii) yields
\begin{align*}
 \|v^{\eps}_{\alpha,\delta}-f^\dagger\|_{L^2(\partial B)}\leq C \|\tilde v^{\eps}_{\alpha,\delta}-\tilde f^\dagger\|_{\Me}\leq C \sqrt{\eps}\|v^\eps_{\alpha,\delta}\|_{\H^\eps}+ C\sqrt{\alpha}\to 0
\end{align*}
as $\delta\to 0$. In particular, $v=f^\dagger=v^\dagger \in {\rm ran}(F)\subset L^2(\partial B)$. Hence, there exists $u\in L^2(\partial H)$ such that $Fu=v$. Lemma~\ref{lem:uniqueness} implies $u=u^\dagger$. The definition of $F$ and unique solvability of \eqref{eq:sharpweakform} implies $v=v^\dagger$ in $D$.
 To show $u_{\alpha,\delta}^\eps\to \tilde u^\dagger$ in $(\H^\eps_\diamond)'$ let $w\in \H^\eps_\diamond$, and let $v^\eps \in \H^\eps_\diamond$ denote the solution to \eqref{eq:diffuseweakform} with right-hand side $\tilde u^\dagger$; cf. Lemma~\ref{lem:perturbation}. Then
 \begin{align*}
  \langle u_{\alpha,\delta}^\eps-\tilde u^\dagger, w\rangle_{\Ue} &= \langle M\nabla (v^\eps_{\alpha,\delta}-v^\eps),\nabla w\rangle_{\omega^\eps} =  \langle M\nabla (v^\eps_{\alpha,\delta}-v^\dagger),\nabla w\rangle_{\omega^\eps}+ \langle M\nabla (v^\dagger-v^\eps),\nabla w\rangle_{\omega^\eps}\\
  &\leq C(\|\sqrt{\omega^\eps}\nabla(v^\eps_{\alpha,\delta}-v^\dagger)\|_{L^2(\Omega)} + \|v^\dagger-v^\eps\|_{\H^\eps}) \|w\|_{\H^\eps}.
 \end{align*}
 In view of Proposition~\ref{prop:convergence1} and Lemma~\ref{lem:perturbation}, the right-hand side vanishes as $\delta\to 0$.
The uniqueness result, Lemma~\ref{lem:uniqueness}, further allows to transfer the convergence to the whole sequence.
\end{proof}
\subsection{Convergence rates}\label{sec:rates}
In order to show convergence rates recall that $u^\dagger$ is the minimum-norm solution of $Fu=f^\dagger$, i.e. a minimizer of
\begin{align*}
 \min \|u\|_{L^2(\partial H)}^2 \quad\text{such that } v_{\mid\partial B}=f^\dagger\quad\text{and}\quad b(u,v;r)=0 \text{ for all } r\in H^1_\diamond(D).
\end{align*}
The associated Lagrangian writes as
\begin{align}\label{eq:Lagrangian}
 L(u,v,\lambda,p)=\|u\|_{L^2(\dH)}^2 - \langle v-f^\dagger,\lambda\rangle+ b(u,v;p).
\end{align}
Assuming that there exists $(\lambda^\dagger,p^\dagger)$ such that $(u^\dagger,v^\dagger,\lambda^\dagger,p^\dagger)$ is a saddle-point of $L$, the following optimality conditions hold true
\begin{align}
 \langle u^\dagger, h_u\rangle_{\partial H}-\langle h_u,p^\dagger\rangle_{\partial H}&=0 \quad\text{for all } h_u\in L^2(\partial H), \label{eq:sad1}\\
 -\langle h_v,\lambda^\dagger\rangle_{\partial B} + \langle M\nabla h_v, \nabla p^\dagger \rangle_D &= 0\quad \text{for all } h_v \in H_\diamond^1(D),\label{eq:sad2}\\
 \langle v^\dagger-f^\dagger,h_\lambda\rangle_{\partial B} &=0\quad \text{for all } h_\lambda\in L^2(\partial B),\label{eq:sad3}\\
 b(u^\dagger,v^\dagger;h_p) &=0\quad \text{for all } h_p\in H_\diamond^1(D).\label{eq:sad4}
\end{align}
Eq. \eqref{eq:sad1} implies $u^\dagger=p^\dagger$ on $\partial H$, where $p^\dagger$ satisfies the adjoint equation \eqref{eq:sad2}, i.e.
\begin{align}\label{eq:source_condition}
 u^\dagger = F^* \lambda^\dagger,
\end{align}
which is the usual source condition.
Vice versa, if \eqref{eq:source_condition} holds true, then \eqref{eq:sad1}--\eqref{eq:sad2} are satisfied,
and $(u^\dagger,v^\dagger,\lambda^\dagger,p^\dagger)$ is a saddle-point of $L$.
In order to simplify the presentation, we will assume that $n(x)$ is an eigenvector of $M(x)$ for $x\in\partial D$, i.e.
\begin{align}\label{eq:M_scalar}
 M(x)n(x)=a(x)n(x)\quad\text{for } x\in\partial D
\end{align}
for some scalar function $a$ satisfying $m\leq a(x)\leq 1/m$ for all $x\in\partial D$ by \eqref{eq:ellipticity}.

\begin{rem}
 Formally, $p^\dagger$ is a solution to 
 \begin{align}\label{eq:p_strong}
  -{\rm div}(M\nabla p^\dagger) = 0 \quad \text{in } D,\quad
   n\cdot M\nabla p^\dagger  = 0 \quad \text{on } \partial H,\quad
   n\cdot M\nabla p^\dagger  = \lambda^\dagger  \quad \text{on } \partial B.
 \end{align}
 Since $n\cdot M\nabla v^\dagger = u^\dagger=p^\dagger$ on $\partial H$ if \eqref{eq:source_condition} holds, regularity assumptions on $u^\dagger$ or $v^\dagger$ can be translated to $p^\dagger$ and $\lambda^\dagger$. Similar to the assumptions on $u^\dagger$ and $v^\dagger$, we will assume that $p^\dagger \in W^{3,\infty}(D)$ in this paper. In particular, $p^\dagger$ is a strong solution to \eqref{eq:p_strong}.
\end{rem}

Assuming \eqref{eq:source_condition} holds true, there exists a saddle-point $(u^\dagger,v^\dagger,\lambda^\dagger,p^\dagger)$ of the Lagrangian defined in \eqref{eq:Lagrangian}.
The error $(u^\eps_{\alpha,\delta}-\tilde u^\dagger, v^\eps_{\alpha,\delta}-v^\dagger, p^\eps_{\alpha,\delta} - \alpha p^\dagger)$ satisfies the saddle-point problem \eqref{saddlepoint} with right-hand side
\begin{align}
   f^\eps(q,w)&= \langle \tilde f^\delta,w\rangle_{\Me} - a^\eps(\tilde u^\dagger,v^\dagger;q,w) - b^\eps(q,w;\alpha p^\dagger), \label{eq:error1}
   \\
   g^\eps(r) &=  - b^\eps(\tilde u^\dagger,v^\dagger;r) \label{eq:error2}
\end{align}
with $(q,w)\in \Ue\times \H^\eps_\diamond$ and $r\in \H^\eps_\diamond$.
In order to obtain error estimates we will estimate the right-hand side of the latter saddle-point problem and employ Theorem~\ref{uniqsoln}.

\begin{lemma}\label{lem:rhs1}
 Let \eqref{eq:noise_level}, \eqref{eq:M_scalar}, and \eqref{eq:source_condition} hold and let $f^\eps$ be defined by \eqref{eq:error1}. Then there exists a constant $C>0$ independent of $\eps$, $\alpha$ and $\delta$ such that 
  \begin{align*}
  \|f^\eps\|_{(\Ue\times\H^\eps_\diamond)'} \leq C (\delta+ \eps^{3/2}\|v^\dagger\|_{W^{3,\infty}(D)}+\eps^{3/2}\alpha^{1/2} \|p^\dagger\|_{W^{3,\infty}(D)}+\alpha \|\lambda^\dagger\|_{L^2(\partial B)}).
 \end{align*} 
\end{lemma}
\begin{proof}
Let $(q,w)\in \Ue\times \H_\diamond^\eps$.
Using the source condition, i.e. $p^\dagger=u^\dagger$ on $\partial H$, we have that
\begin{align*}
 f^\eps(q,w) = \langle \tilde f^\delta-\tilde f^\dagger + \tilde v^\dagger - v^\dagger, w\rangle_{\Me}-\alpha \langle M\nabla w,\nabla p^\dagger\rangle_{\omega^\eps} +\alpha \langle p^\dagger-\tilde p^\dagger, q\rangle_{\Ue}.
\end{align*}
Using \eqref{eq:noise_extended}, Cauchy-Schwarz inequality and Lemma~\ref{lem:error_diffuse_integral} (iii) we obtain
\begin{align*}
  \langle \tilde f^\delta-\tilde f^\dagger, w\rangle_{\Me} + \langle v^\dagger-\tilde v^\dagger, w \rangle_{\Me}
  &\leq C (\delta+ \eps^{3/2} \|v^\dagger\|_{W^{2,2}(D)}) \|w\|_{\Me},
\end{align*}
where we used $\partial_n v^\dagger=0$ on $\partial B$ by \eqref{eq:M_scalar}. Since $\partial_n p^\dagger=0$ on $\partial H$ by \eqref{eq:M_scalar} and \eqref{eq:p_strong}, we similarly obtain with Lemma~\ref{lem:error_diffuse_integral} (iii) 
\begin{align*}
 \langle p^\dagger-\tilde p^\dagger, q \rangle_{\Ue}&\leq  C \eps^{3/2} \|p^\dagger\|_{W^{2,2}(D)} \|q\|_{\Ue}.
\end{align*}
Integration by parts and $-\nabla\omega^\eps=n|\nabla\omega^\eps|$ yield
\begin{align*}
  \langle M\nabla w, \nabla p^\dagger\rangle_{\omega^\eps} &= -\langle{\rm div}(M \nabla p^\dagger), w\rangle_{\omega^\eps} + \langle n\cdot M\nabla p^\dagger, w\rangle_{\Me}+ \langle n\cdot M\nabla p^\dagger, w\rangle_{\Ue}
\end{align*}
An application of Lemma~\ref{lem:error_diffuse_integral} (i) yields
\begin{align*}
  \langle{\rm div}(M \nabla p^\dagger), w\rangle_{\omega^\eps}\leq C \eps^{3/2}\|p^\dagger\|_{W^{3,\infty}(\Omega)} \|w\|_{\H^\eps},
\end{align*}
and, since $n\cdot M\nabla p^\dagger =0$ on $\partial H$, Lemma~\ref{lem:error_diffuse_integral} (iv) gives
\begin{align*}
 \langle n\cdot M\nabla p^\dagger, w\rangle_{\Ue} \leq C \eps^{3/2} \|p^\dagger\|_{W^{3,2}(\Omega;\omega^\eps)} \|w\|_{\H^\eps},
\end{align*}
as well as, using $n\cdot M\nabla p^\dagger=\lambda^\dagger$ on $\partial B$ and Lemma~\ref{lem:extension},
\begin{align*}
\langle n\cdot M\nabla p^\dagger, w\rangle_{\Me} &=  \langle n\cdot M\nabla p^\dagger -E_B(n\cdot M\nabla p^\dagger), w\rangle_{\Me} +\langle E_B \lambda^\dagger, w\rangle_{\Me} \\
&\leq C (\eps^{3/2}  \|p^\dagger\|_{W^{3,2}(\Omega;\omega^\eps)} \|w\|_{\H^\eps} + \|\lambda^\dagger\|_{L^2(\partial B)} \|w\|_{\Me}).
\end{align*}
Collecting the above estimates and using the definition of $\|(q,w)\|_\alpha$ yields the assertion.
\end{proof}

Using Lemma~\ref{lem:rhs1} and Lemma~\ref{lem:perturbation}, we infer from Theorem~\ref{uniqsoln} the following error estimate.
\begin{thm}\label{thm:error}
Let $0<\alpha\leq \alpha_0$ and $\eps>0$. Moreover, let \eqref{eq:noise_level}, \eqref{eq:M_scalar} and \eqref{eq:source_condition} hold. Then there exists $C>0$ independent of $\eps$ and $\alpha$ such that
\begin{align*}
  \alpha \|u^\eps_{\alpha,\delta}-\tilde u^\dagger\|_{\Ue}^2 + \alpha\| \nabla v^\eps_{\alpha,\delta}-\nabla v^\dagger\|_{L^2(\omega^\eps)}^2 + \|v^\eps_{\alpha,\delta}-v^\dagger\|_{\Me}^2 + \|p^\eps_{\alpha,\delta}-\alpha p^\dagger\|_{\H^\eps}^2\\
     \leq C \big(\delta^2+ \eps^3\|v^\dagger\|^2_{W^{3,\infty}(D)}+\eps^3\alpha \|p^\dagger\|^2_{W^{3,\infty}(D)}+\alpha^2 \|\lambda^\dagger\|_{L^2(\partial B)}^2\big).
\end{align*}
\end{thm}

With an appropriate choice of $\eps$ and $\alpha$ in terms of $\delta$ we obtain the overall optimal order of convergence:

\begin{cor}\label{cor:convergence_rates}
Let the assumptions of Theorem~\ref{thm:error} hold true. For the a priori choice $\alpha\approx\delta$ and $\eps\approx\delta^{2/3}$ we obtain the following convergence rates
\begin{align}\label{eq:convergence_rates}
 \|u^\eps_{\alpha,\delta}-\tilde u^\dagger\|_{\Ue} + \|\nabla v^\eps_{\alpha,\delta}-\nabla v^\dagger\|_{L^2(\omega^\eps)} = O(\sqrt{\delta})\quad\text{and}\quad
 \|v^\eps_{\alpha,\delta}-v^\dagger\|_{\Me} = O(\delta).
\end{align}
\end{cor}

\begin{rem}\label{rem:lower_regularity}
  If $v^\dagger,p^\dagger\in W^{1,\infty}(D)$ only, we have to replace $\eps^3$ in the previous estimates by $\eps$, cf. Lemma~\ref{lem:error_diffuse_integral}. The choice $\alpha\approx\delta$ and $\eps\approx\delta^2$ then yields \eqref{eq:convergence_rates}.
\end{rem}

\begin{rem}\label{rem:different_direction}
 Assumption \eqref{eq:M_scalar} can be bypassed, if one defines the extension off the interface $E_{Mn}v$ to be constant along the straight line $t\mapsto x+tM(x)n(x)$, $x\in\partial D$. Moreover, the estimates in the appendix have to be adapted in a similar way.
\end{rem}


We finally mention that a generalization of \eqref{eq:source_condition} to more general source conditions of the form 
$u^\dagger = (F^*F)^\mu \lambda^\dagger$  (with $0 < \mu \leq 1$) can be carried out in a similar way. The main change then concerns the last two terms on the right-hand side of the estimate in Lemma \ref{lem:rhs1}, which yield different orders in terms of $\alpha$. Interestingly the optimal choice $\eps^3 \approx \delta^2$ is unaffected by the specific source condition. 

\section{Numerical Solution}

For the numerical solution we discretize the saddle-point system \eqref{saddlepoint} with standard piecewise linear finite element methods on triangular grids not resolving the interface but adaptively refined based on the gradient of $\varphi^\eps$. Note that this is equivalent to the optimality system for a direct finite element discretization of the minimization problem for \eqref{eq:Tikhonov_diffuse}. In the following we discuss some further aspects arising in the solution of the linear system.

\subsection{Preconditioning of the Saddle-point System}
In order to solve the saddle-point system \eqref{saddlepoint} in reasonable time, we rely on efficient preconditioners. 
We concluded that all the constants in the stability estimates were 
independent of the parameter $\eps$, cf. Lemma \ref{lem:continuity} and Theorem \ref{uniqsoln}. Consequently, to obtain an
$\eps$-robust preconditioner becomes a matter of applying the proper Riesz maps, denoted by $R_{\mathcal{U}^\eps}: \mathcal{U}^\eps
\rightarrow (\mathcal{U}^\eps)'$ and $R_{\mathcal{H}_\diamond^\eps}: \mathcal{H}_\diamond^\eps \rightarrow (\mathcal{H}_\diamond^\eps)'$.
Furthermore, let us introduce the operators
\begin{eqnarray*}
 Q^\eps&:& \mathcal{U}^\eps \rightarrow (\mathcal{H}_\diamond^\eps)', 
 \quad u \mapsto -\langle u, w\rangle_{\Ue},\\ 
 P^\eps&:& \mathcal{H}_\diamond^\eps \rightarrow (\mathcal{H}_\diamond^\eps)',
 \quad v \mapsto \langle M \nabla v, \nabla w\rangle_{\omega^\eps},\\
 {T^\eps} &:& \mathcal{H}_\diamond^\eps \rightarrow (\mathcal{H}_\diamond^\eps)', 
 \quad v \mapsto \langle v, w\rangle_{\Me}, \\
 {\tilde{T}^\eps} &:& \mathcal{M}^\eps \rightarrow (\mathcal{H}_\diamond^\eps)',
 \quad f \mapsto \langle f, w \rangle_{\Me},
 \end{eqnarray*}
with $w\in \H^\eps_\diamond$. Using these operators, we can write \eqref{saddlepoint} in the form 
\begin{equation}\label{kkt}
\underbrace{\begin{bmatrix} \alpha R_{\mathcal{U}^\eps}&0&[Q^\eps]'\\ 0& T^\eps&[P^\eps]' \\ Q^\eps&P^\eps&0 
\end{bmatrix}}_{\widehat{\mathcal{A}}_{\alpha}^\eps}
\underbrace{\begin{bmatrix} u^\eps \\ v^\eps \\ p^\eps \end{bmatrix}}_{q^\eps} = 
\underbrace{\begin{bmatrix} 0 \\ \tilde{T}^\eps f \\ 0 \end{bmatrix}}_b
,\end{equation}
where we have 
\begin{equation}
 {\widehat{\mathcal{A}}}_{\alpha}^\eps: \mathcal{U}^\eps \times \mathcal{H}_\diamond^\eps
 \times \mathcal{H}_\diamond^\eps \rightarrow (\mathcal{U}^\eps)' \times (\mathcal{H}_\diamond^\eps)'
 \times (\mathcal{H}_\diamond^\eps)'.
\end{equation}
Since this operator $\widehat{\mathcal{A}}_\alpha^\eps$ maps from a (product) Hilbert space onto its dual space, 
Krylov subspace methods are not readily available. However, assuming that an operator $\mathcal{B}^\eps:
(\mathcal{U}^\eps)' \times (\mathcal{H}^\eps)' \times (\mathcal{H}^\eps)' \rightarrow
\mathcal{U}^\eps \times \mathcal{H}^\eps \times \mathcal{H}^\eps$ is available,
Krylov subspace methods can be employed to solve
\begin{equation}
 \mathcal{B}^\eps\widehat{\mathcal{A}}_\alpha^\eps q^\eps = \mathcal{B}^\eps b. \nonumber
\end{equation}
To obtain an efficient solution, the preconditioner $\mathcal{B}^\eps$ must be an isomorphism, see \cite{Mar11}. We
propose to apply inverse Riesz maps to derive such a preconditioner, which lead to the preconditioned system
\begin{equation}\label{prec_kkt} 
\underbrace{\begin{bmatrix} R_{\mathcal{U}_\beta^\eps}^{-1} & 0 & 0 \\ 
 0 & R_{\mathcal{H}_\diamond^\eps}^{-1} & 0 \\ 0 & 0 & R_{\mathcal{H}_\diamond^\eps}^{-1} 
\end{bmatrix}}_{\mathcal{B}^\eps}
\underbrace{\begin{bmatrix} \alpha R_{\mathcal{U}_\beta^\eps}&0&[Q^\eps]'\\ 0& T^\eps&[P^\eps]' \\ Q^\eps&P^\eps&0 
\end{bmatrix}}_{{\hat{\mathcal{A}}}_{\alpha}^\eps}
\underbrace{\begin{bmatrix} u^\eps \\ v^\eps \\ p^\eps \end{bmatrix}}_{q^\eps}   
= \begin{bmatrix} R_{\mathcal{U}_\beta^\eps}^{-1} & 0 & 0 \\ 
 0 & R_{\mathcal{H}_\diamond^\eps}^{-1} & 0 \\ 0 & 0 & R_{\mathcal{H}_\diamond^\eps}^{-1} 
\end{bmatrix}\underbrace{\begin{bmatrix} 0 \\ \tilde{T}^\eps f \\ 0 \end{bmatrix}}_b.
\end{equation}
We observe that
\begin{equation}\label{eq:ainv}
 \mathcal{A}_\alpha^\eps = \mathcal{B}^\eps\widehat{\mathcal{A}}_\alpha^\eps: \mathcal{U}^\eps \times \mathcal{H}_\diamond^\eps
 \times \mathcal{H}_\diamond^\eps \rightarrow \mathcal{U}^\eps \times \mathcal{H}_\diamond^\eps \times \mathcal{H}_\diamond^\eps,
\end{equation}
and consequently, since ${\mathcal{A}}_\alpha^\eps$ is a symmetric indefinite operator,
the MINRES algorithm can be applied to solve the optimality system.

\begin{rem}
 For our numerical examples we will use a norm induced by the inner product
 \begin{equation}
  \langle M\nabla v,\nabla v\rangle_{\omega^\eps}+\langle  v,v\rangle_{\omega^\eps}
 \end{equation}
 on $\mathcal{H}_\diamond^\eps$. This influences the preconditioner $\mathcal{B}^\eps$, 
 resulting in a slightly different stiffness matrix from the discretization of the Riesz map $R_{\mathcal{H}_\diamond^\eps}$. 
 From a numerical investigation, this gave better iteration counts, and we therefore apply this alternative norm in the numerical section.
\end{rem}

\subsection{Spectrum of the preconditioned system}\label{sec:spectrum}

Operators similar to $\mathcal{A}_\alpha^\eps$ were thoroughly analyzed in \cite{Nie13}. Under given assumptions, an efficient and robust solution of the 
saddle-point system \eqref{prec_kkt} can be guaranteed. More specifically, the authors of \cite{Nie13} show that for a sound discretization of $\mathcal{A}^\eps_\alpha$
defined in \eqref{prec_kkt}-\eqref{eq:ainv}, the spectrum of the associated discretized operator $\mathcal{A}_\alpha^{\eps,h}$ satisfied
%
\begin{equation}\label{spectrum}
\textnormal{sp}(\mathcal{A}_\alpha^{\eps,h}) \subset [-b,-a] \cup [c\alpha,2\alpha] \cup 
\{ \tau_1, \tau_2, ..., \tau_{N(\alpha)\}} \cup [a,b],  
\end{equation}
where $N(\alpha) = O(\ln(\alpha^{-1}))$ and the constants $a,\,b,\,c$ are independent of $\alpha$ (and here also of $\eps$).

To guarantee this spectrum, the following assumptions must be satisfied:
\begin{list}{}{}
\item $\mathcal{A}\mathbf{1:} \ P^\eps: \mathcal{H}_\diamond^\eps \rightarrow (\mathcal{H}_\diamond^\eps)'$ is 
      bounded, linear, and invertible.
\item $\mathcal{A}\mathbf{2:} \ Q^\eps: \mathcal{U}_\beta^\eps \rightarrow (\mathcal{H}_\diamond^\eps)'$ 
      is bounded and linear. 
\item $\mathcal{A}\mathbf{3:} \ T^\eps: \mathcal{H}_\diamond^\eps \rightarrow (\mathcal{H}_\diamond^\eps)'$ is bounded and linear. 
\item $\mathcal{A}\mathbf{4:}$ The operator equation \eqref{eq:diffuseIP} is ill-posed. 
\end{list}

Assumptions $\mathcal{A}\mathbf{1}$-$\mathcal{A}\mathbf{4}$ follow immediately from the analysis in
Section~\ref{sec:analysis}.

\subsection{Implementation}

We implemented the code using cbc.block, which is a FEniCS-based Python
implemented library for block operators. See \cite{Simula.simula.1112} for details.
The PyTrilinos package was used to compute an approximation of the preconditioner $\mathcal{B}^\eps$ in \eqref{prec_kkt}. 
We approximated $\mathcal{B}^\eps$ using AMG with a symmetric Gau\ss-Seidel smoother with three smoothing sweeps. 
All tables containing iteration counts for the MINRES method were generated with this approximate preconditioner.  
On the other hand, the eigenvalues of $\mathcal{A}_\alpha^\eps = \mathcal{B}^\eps\widehat{\mathcal{A}}_\alpha^\eps$
were computed with the \textit{exact} preconditioner $\mathcal{B}^\eps$ in Octave. 
The MINRES iteration process was stopped as soon as   
  \begin{equation}\label{eq:stop} \frac{\|r_n\|}{\|r_0\|} =
   \frac{\|\mathcal{B}^\eps [\hat{\mathcal{A}}^\eps_\alpha
       q_n - b]\|_{\mathcal{U}^\eps \times \mathcal{H}^\eps \times \mathcal{H}^\eps} }  
       {\|\mathcal{B}^\eps[\hat{\mathcal{A}}^\eps_\alpha q_0 - b]\|_{\mathcal{U}^\eps \times \mathcal{H}^\eps \times \mathcal{H}^\eps}}
     < \rho.  
    \end{equation} 
   Here, $\rho$ is a small positive parameter.  
The exact data $u^\dagger$ was computed from an appropriate source condition, 
i.e. $F^*w = u^\dagger$, for some $w \in L^2(\partial B)$. Then, we computed $Fu^\dagger= f^\dagger$. Noise was then added to $f^\dagger$, 
and the noisy data was extended to ${\rm supp}(\gamma_B |\nabla\omega^\eps|)$ by the extension operator $E_B$, see Section~\ref{sec:extensions}.

\subsection{Examples}

In our simulations, we use a ``circle in circle'' domain. The domain $D$ is defined as
\begin{equation}
 D = \{(x,y) \in \mathbb{R}^2: 0.3 < \sqrt{x^2 + y^2} < 1\}. \nonumber 
\end{equation}
The diffuse domain $D_\eps$ is then simply the scaling
\begin{equation}
 D_\eps = \{(x,y) \in \mathbb{R}^2: 0.3 - \eps < \sqrt{x^2 + y^2} < 1 + \eps\}. \nonumber
\end{equation}
Furthermore, the conductivity tensor $M(x,y)$ is defined as
\begin{equation*}
 M = \bar L \Sigma \bar L^\top,
\end{equation*}
where
\begin{equation*}
 \bar L = \frac{1}{\|(x,y)\|}\begin{bmatrix} y & x \\ -x & y \end{bmatrix}, \qquad
 \Sigma = \begin{bmatrix} 1 & 0 \\ 0 & 0.3 \end{bmatrix}.
\end{equation*}
One easily verifies that \eqref{eq:M_scalar} holds for this choice of $M$.
In Table \ref{tab:conv_criterion}, we see the iteration numbers for different values of $\alpha$ and $\eps$. 
As expected, there is no dependency on the diffuse domain parameter $\eps$, cf. Section \ref{sec:spectrum}. Furthermore, for the regularization parameter $\alpha$, 
we get the expected logarithmic growth in iteration numbers when $\alpha \rightarrow 0$. For example, when $\eps = 2^{-6}$, 
the growth is well modeled by the function
\begin{equation}
 \alpha\mapsto 55 - 24\log_{10}(\alpha). \nonumber
\end{equation}

\begin{table}[ht] \centering
    \begin{tabular}{ | l || l | l | l | l | l |}
      \hline
      $\eps$ \textbackslash $\alpha$ & 1 & .1 & .01 & .001 & .0001 \\ \hline \hline
      $2^{-2}$ & 57 & 100 & 143 & 186 & 238   \\ \hline
      $2^{-3}$ & 57 & 91  & 126 & 157 & 195  \\ \hline
      $2^{-4}$ & 64 & 102 & 126 & 144 & 183   \\ \hline
      $2^{-5}$ & 57 & 83  & 115 & 143 & 159  \\ \hline
      $2^{-6}$ & 55 & 79  & 105 & 123 & 155  \\ \hline
    \end{tabular}
    \caption{The number of MINRES iteration required to solve the discretized system associated with 
    \eqref{prec_kkt}. The stopping criterion $\rho = 10^{-10}$, see \eqref{eq:stop}.}\label{tab:conv_criterion}
\end{table}

Figure \ref{fig:eig_ex} shows the eigenvalues of $\mathcal{A}_\alpha$. The band structure is in accordance with the 
analysis in \cite{Nie13}, with three bands of eigenvalues, and a limited number of isolated eigenvalues.

\begin{figure}[ht]{}
    \centering
    \includegraphics[scale=0.5]{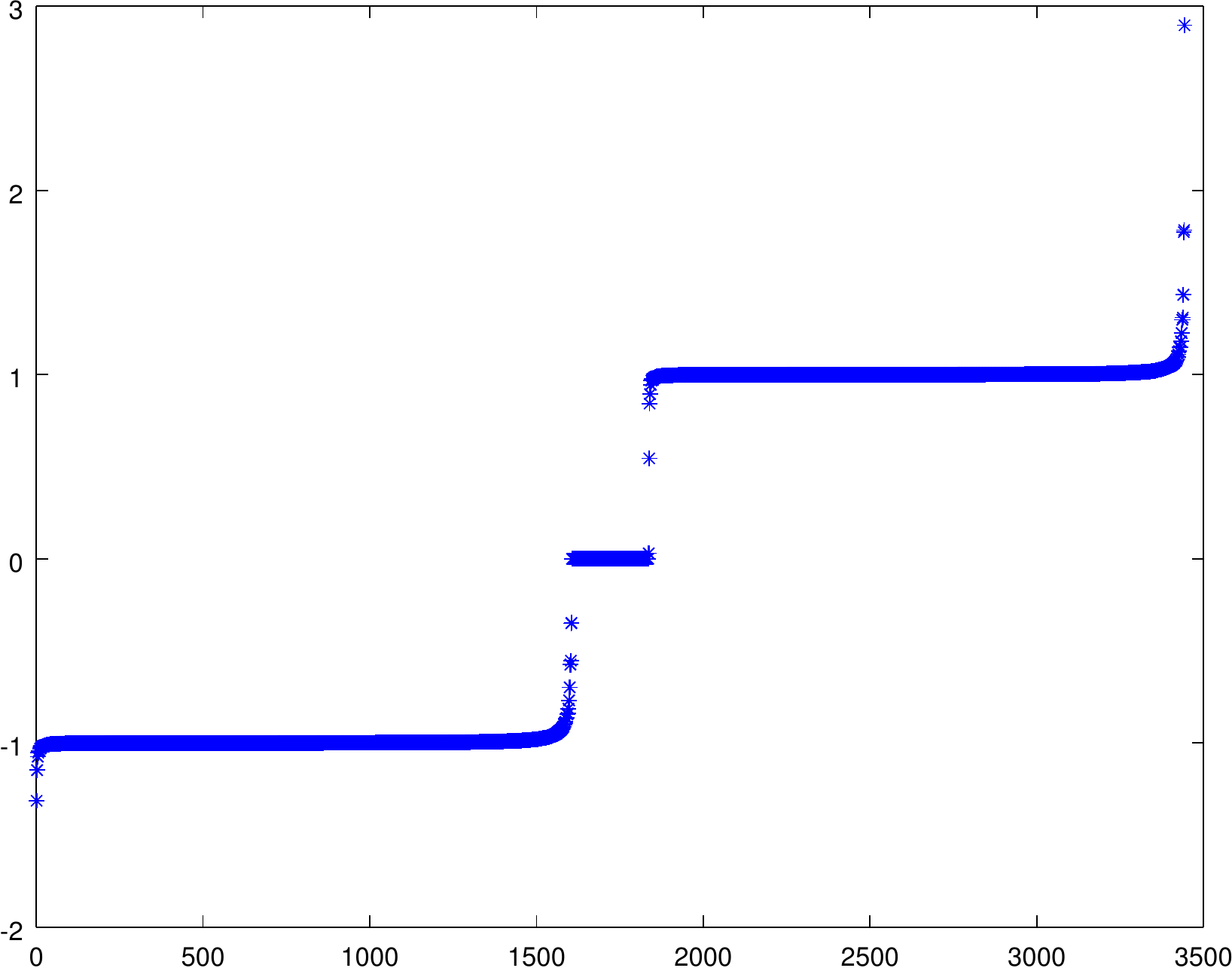}
    \caption{Plot of the eigenvalues associated with $\mathcal{A}_\alpha^\eps$ in Example 1. 
    Here $\alpha = 10^{-4}$ and $\eps = 0.125$. The eigenvalues are computed on a course mesh with 1\,605 vertices.} \label{fig:eig_ex}
\end{figure}

We recall Assumption $\mathcal{A}\mathbf{4}$, i.e. that the operator equation \eqref{eq:diffuseIP} is ill-posed. In Figure \ref{fig:eig_vals1}, 
logarithmic plots of the absolute values of the eigenvalues of $\mathcal{A}^\epsilon_0$ are displayed. The clustering of eigenvalues
around 0 is an effect of the ill-posed nature of \eqref{eq:diffuseIP}.
\begin{figure}
  \begin{subfigure}[ht]{0.45\textwidth}
    \centering
    \includegraphics[scale=0.5]{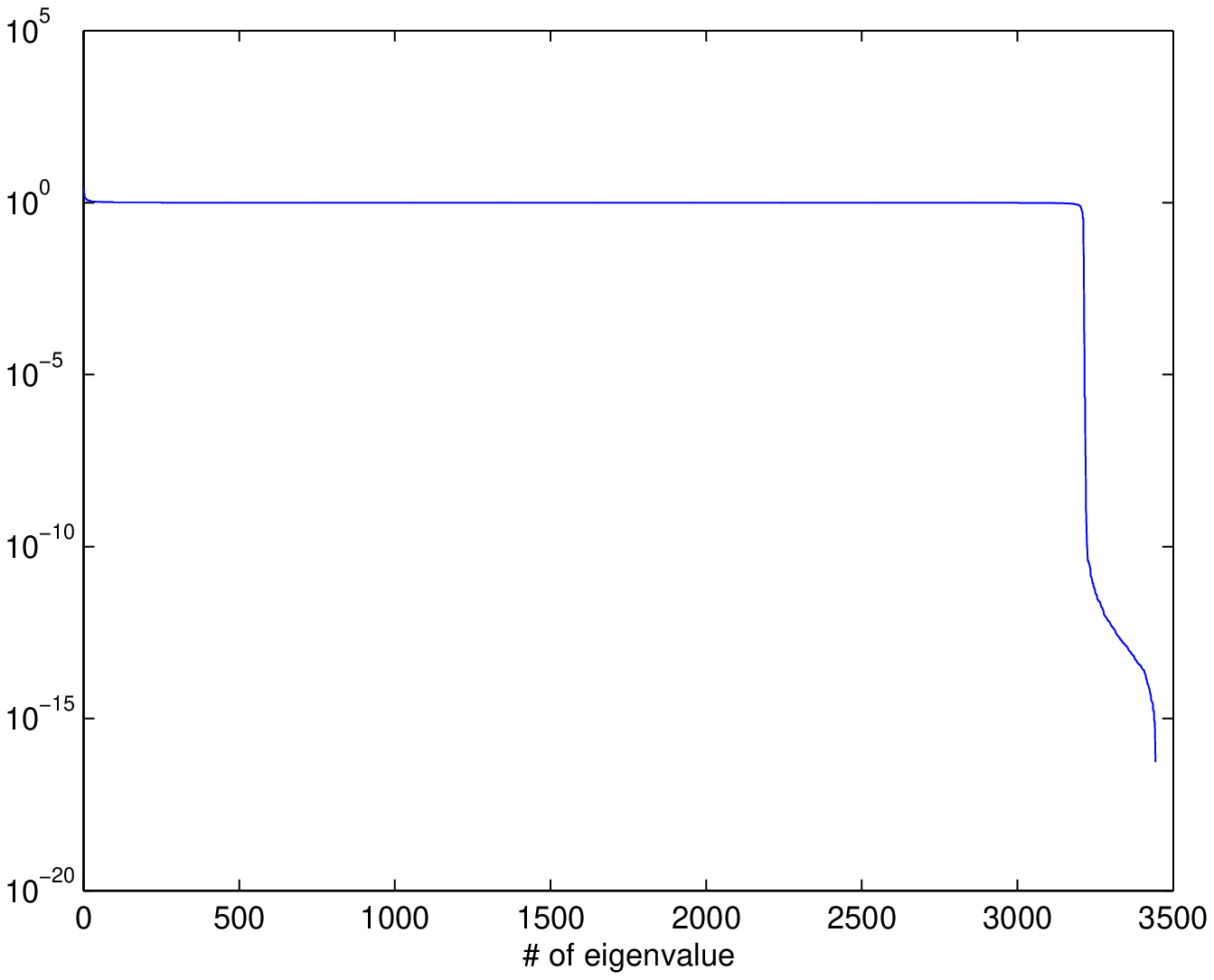}
    \caption{All eigenvalues.}
  \end{subfigure}
  \begin{subfigure}[ht]{0.45\textwidth}
    \centering
    \includegraphics[scale=0.5]{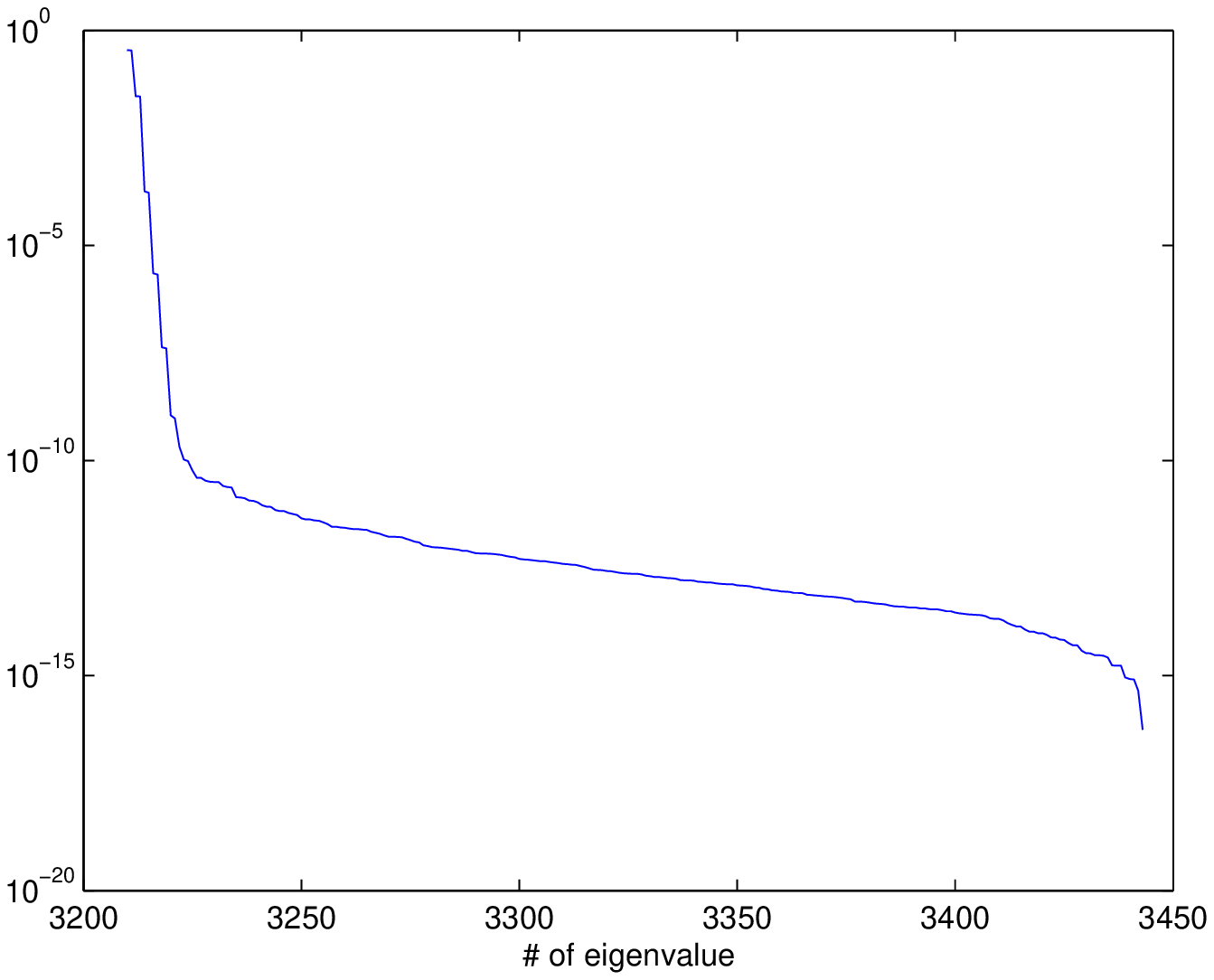}
    \caption{Zoomed in on the smallest eigenvalues.}
  \end{subfigure}
  \caption{Logarithmic plots of the absolute values of the eigenvalues of $\mathcal{A}_0^\eps$.}\label{fig:eig_vals1}
\end{figure}

From a practical point of view, we are concerned with the performance of the diffuse domain method in comparison to the 
standard inverse formulation, i.e. with the optimization performed on the exact domain. We will compare the solutions both visually and in norm sense. 

In Figure \ref{fig:compCtrl}, the exact source function is displayed along with inverse solutions on both the exact and diffuse mesh.
Similar comparisons are displayed in Figures \ref{fig:compState} and \ref{fig:compAdj} for the state and adjoint functions, respectively.
The functions defined on a surface, i.e. either on $\partial H$ or $\partial B$, are extended by the appropriate constant extension operator, see Section \ref{sec:extensions}.

For the control functions, the inverse solution $u_{\alpha,\delta}$ displayed in Figure \ref{fig:compCtrl}b) is visually identical to the exact source function $u^\dagger$.
These are also visually identical to $u_{\alpha,\delta}^\eps$ displayed in Figure \ref{fig:compCtrl}c), where $\eps = 0.03125 = \sqrt{\delta}$. With a larger
choice of $\eps$, however, the solution is quite different from the source $u^\dagger$, see Figure \ref{fig:compCtrl}d) where $\eps = 1/4$.
\begin{figure}
\centerline{%
\begin{tabular}{c@{\hspace{1pc}}c}
\includegraphics[scale=.15]{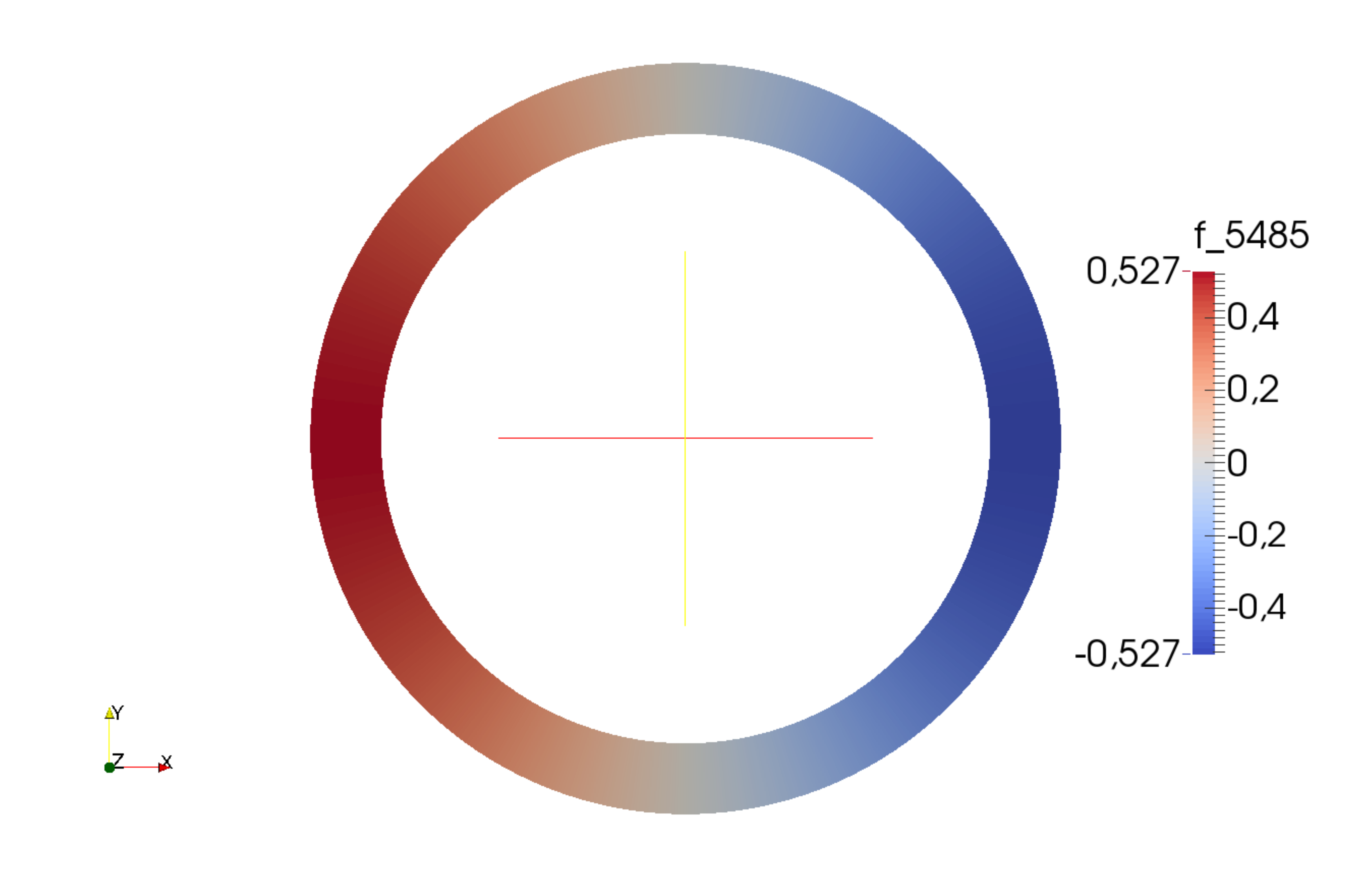} &
\includegraphics[scale=.15]{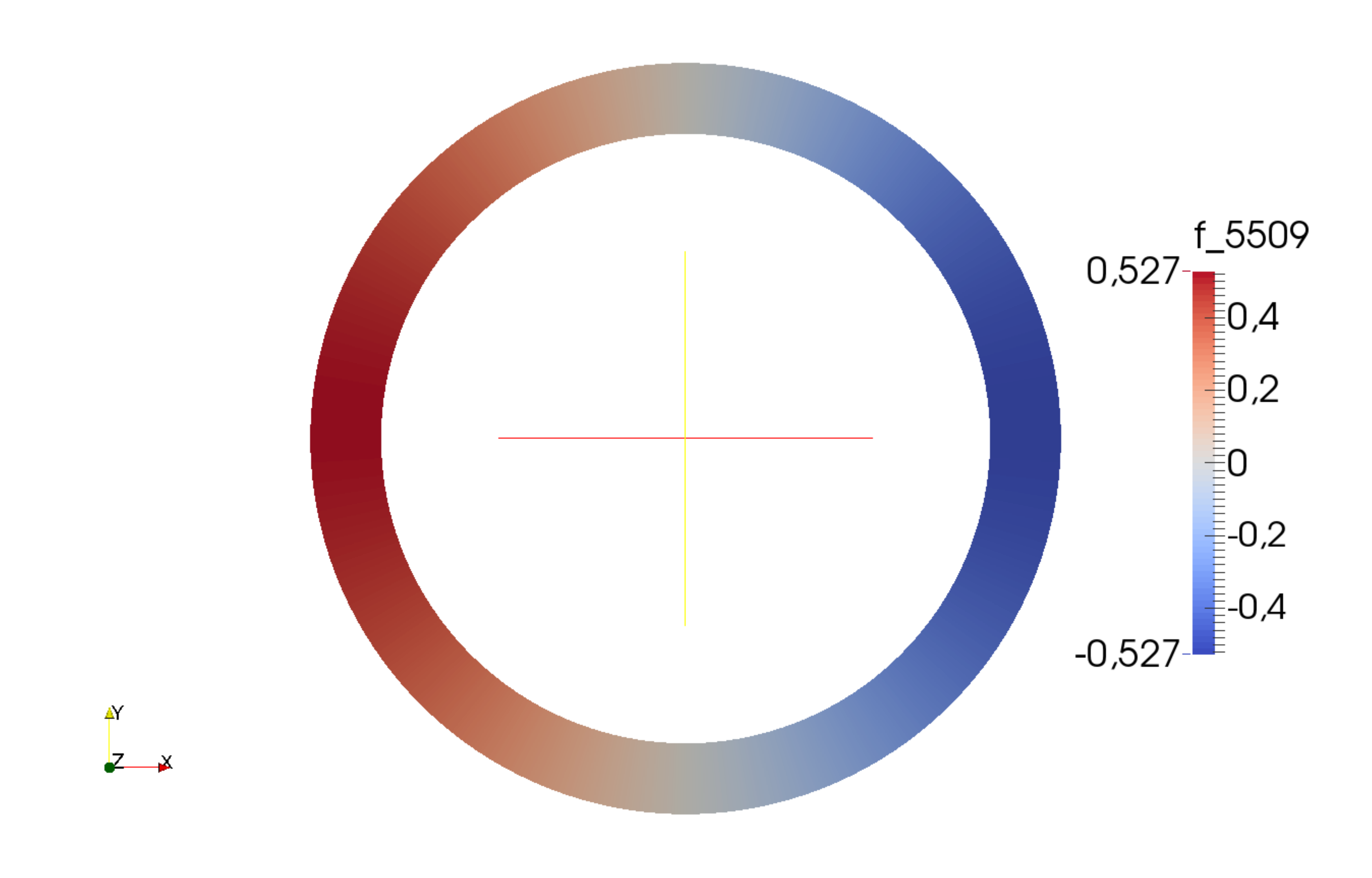} \\
(a)~~ The input source $\tilde u^\dagger$. & (b)~~ Inverse solution $\tilde u_{\alpha,\delta}$ on the exact mesh $D$. \\
\includegraphics[scale=.15]{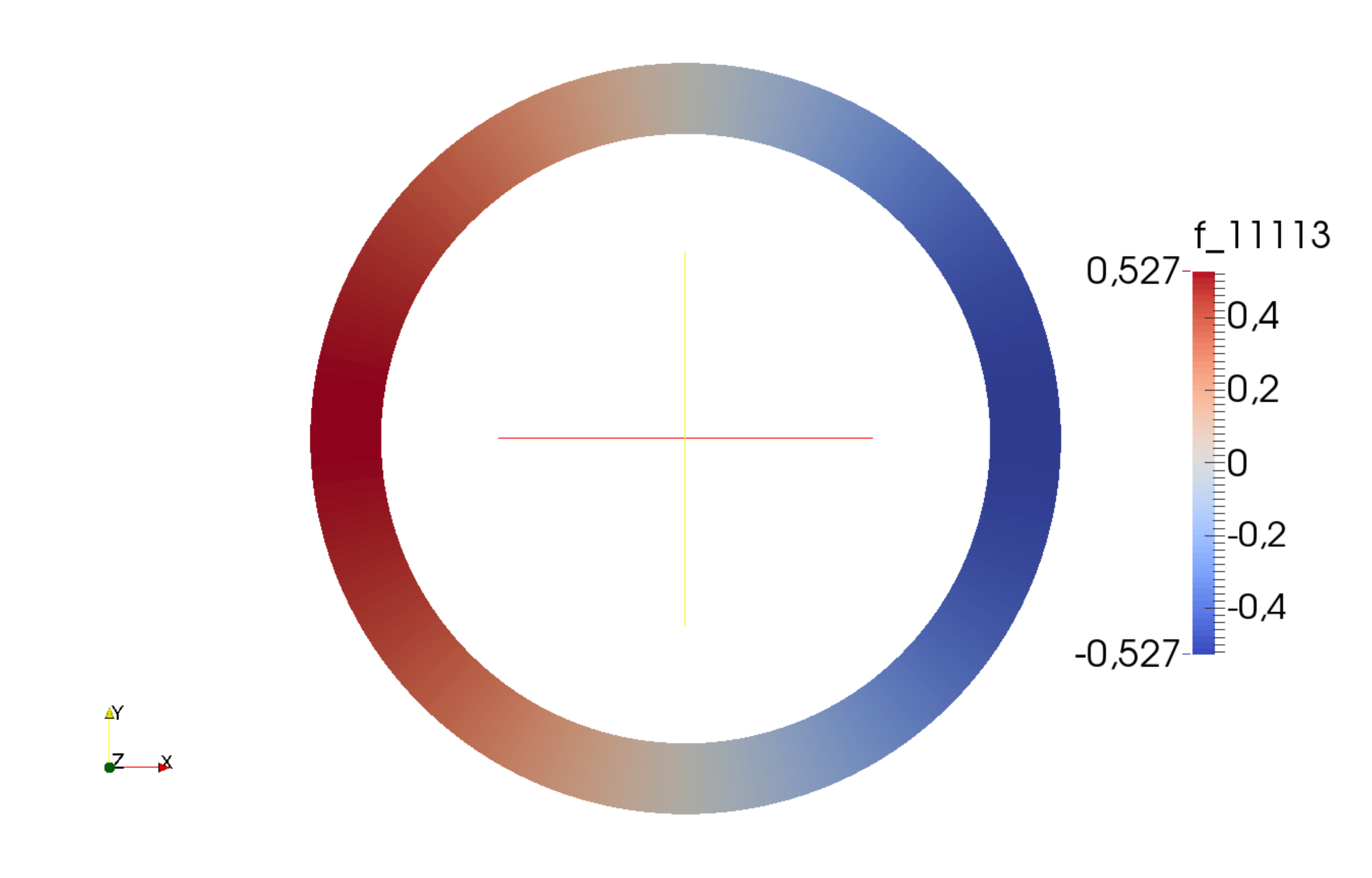} &
\includegraphics[scale=.15]{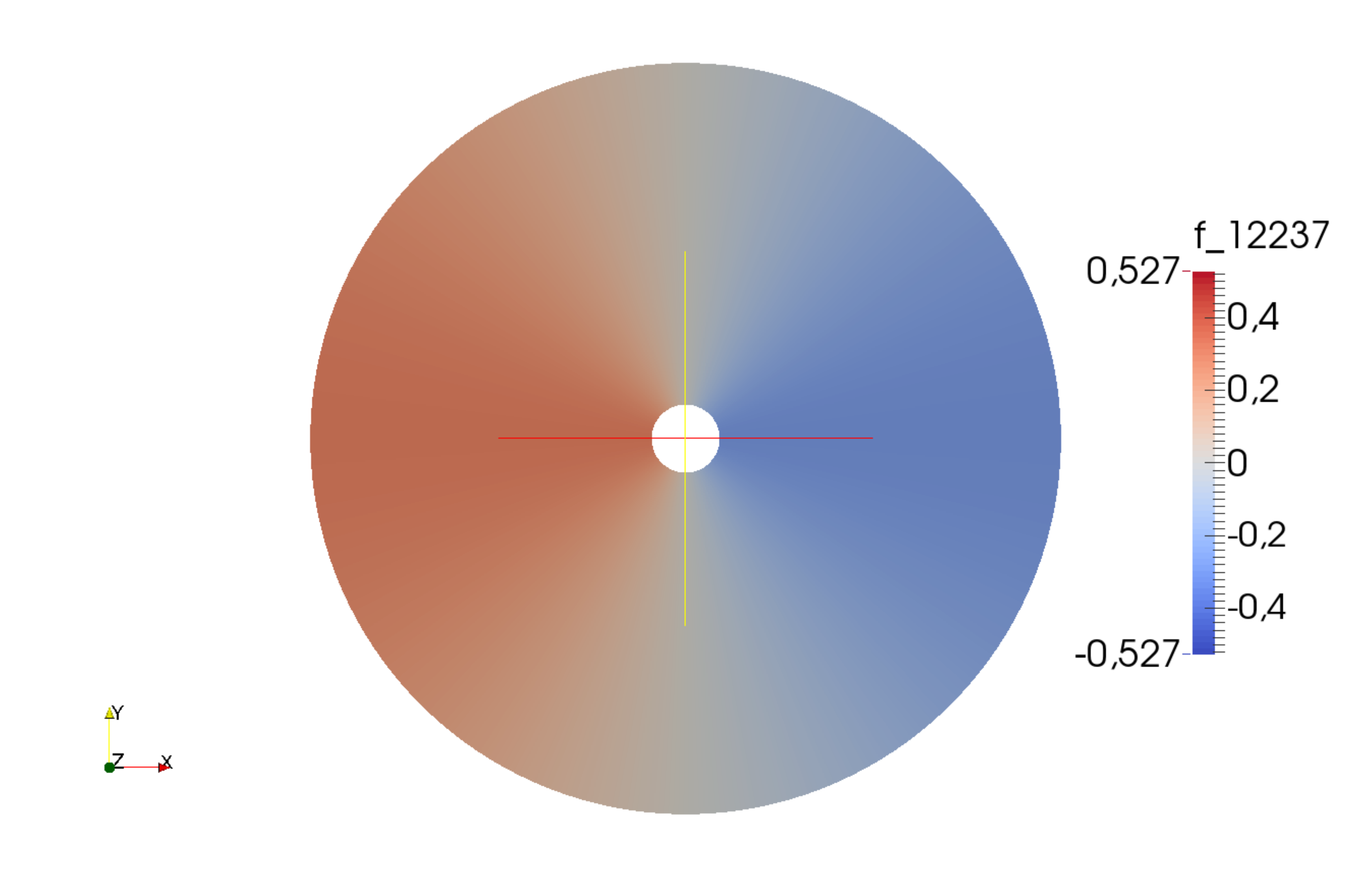} \\
(c)~~ Diffuse solution $u_{\alpha,\delta}^\eps$ for $\eps = 0.03125 = \delta^{1/2}$. & (d)~~ Diffuse solution $u_{\alpha,\delta}^\eps$ for $\eps = 1/4$. 
\end{tabular}}
\caption{A comparison of different control functions and the input source in a). In a) and b), the control is only defined on $\partial H$, so we therefore applied the constant 
extension $E_H$ for the visualization, see Section \ref{sec:extensions}. 
In b), c) and d), $\delta =  2^{-10}$ and $\alpha = \delta/2$.}\label{fig:compCtrl}
\end{figure}
If we consider the state functions, the choice of $\eps$ is less important. All solutions displayed in Figure \ref{fig:compState} 
are basically identical from a visual perspective.
For the adjoint functions, there seem to be some visual difference between $p_{\alpha,\delta}$ and $p_{\alpha,\delta}^\eps$,
i.e. for the adjoint on the exact mesh and on the diffuse mesh for $\eps = 0.03125$, but the order of magnitude of these functions
is only $10^{-3}$.

\begin{figure}
\centerline{%
\begin{tabular}{c@{\hspace{1pc}}c}
\includegraphics[scale=.15]{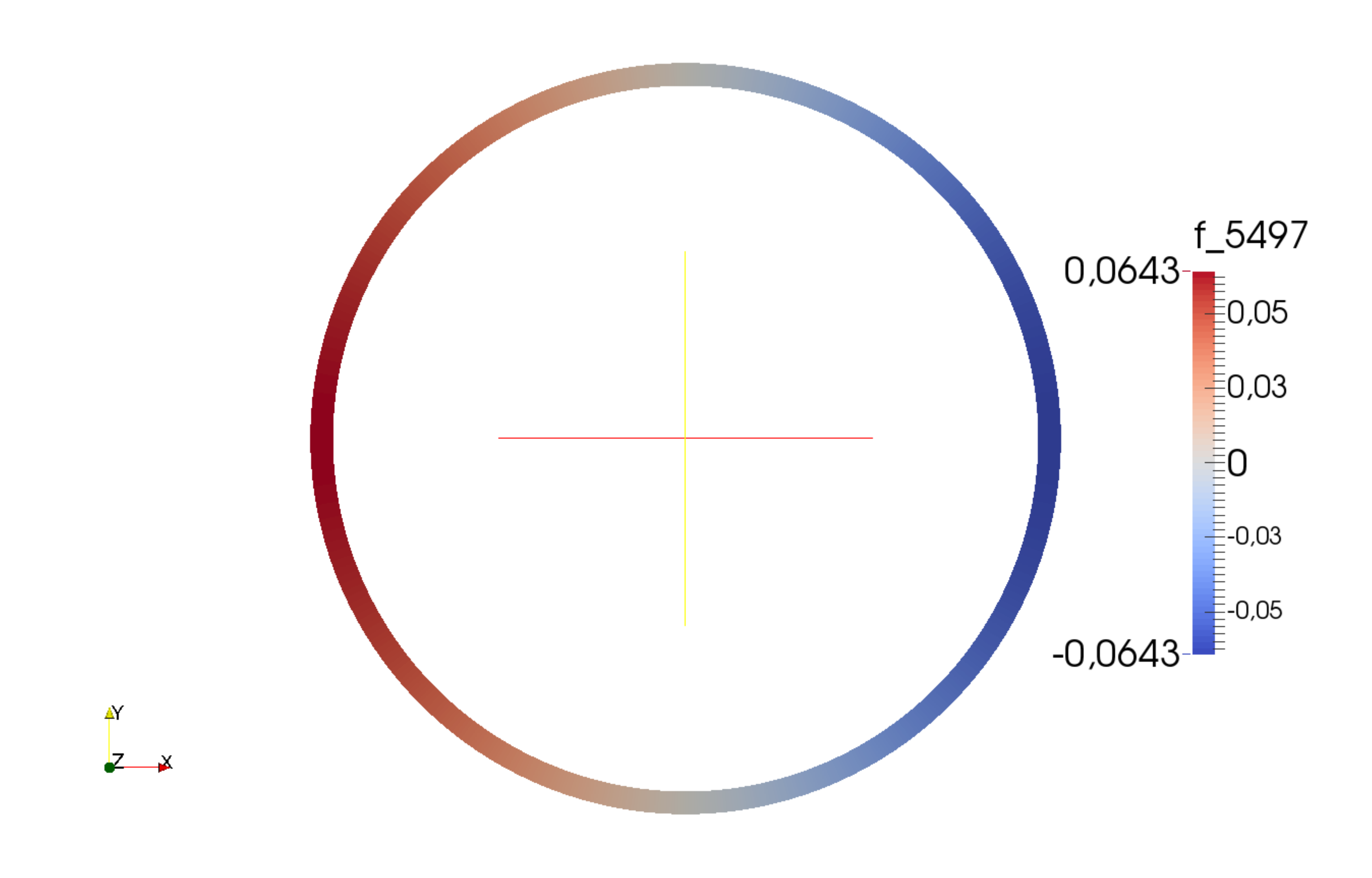} &
\includegraphics[scale=.15]{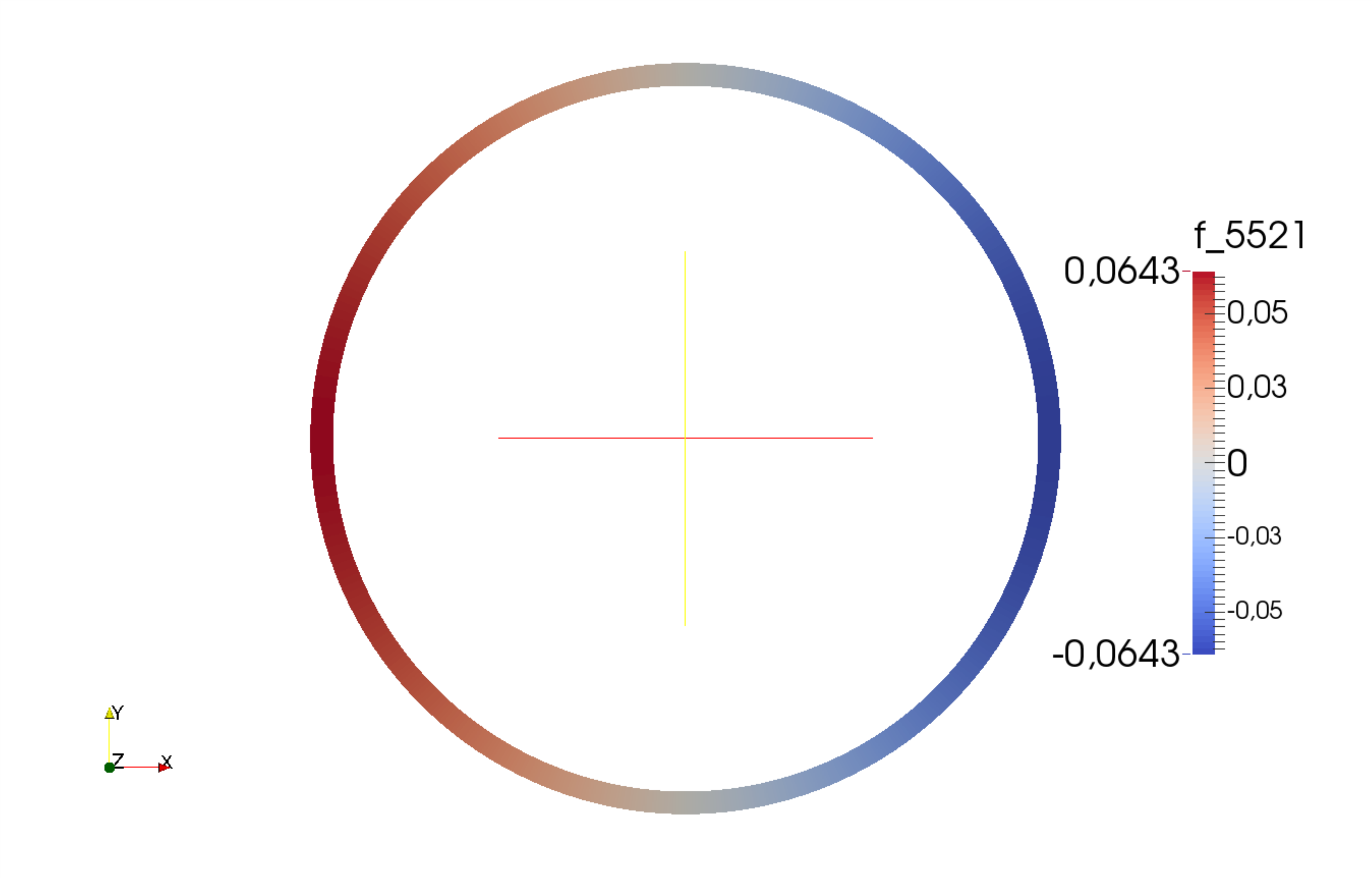} \\
(a)~~ The exact data $\tilde f^\dagger$. & (b)~~ The state $\tilde f_{\alpha,\delta}= E_B(v_{\alpha,\delta|\partial B})$ on the exact mesh $D$. \\
\includegraphics[scale=.15]{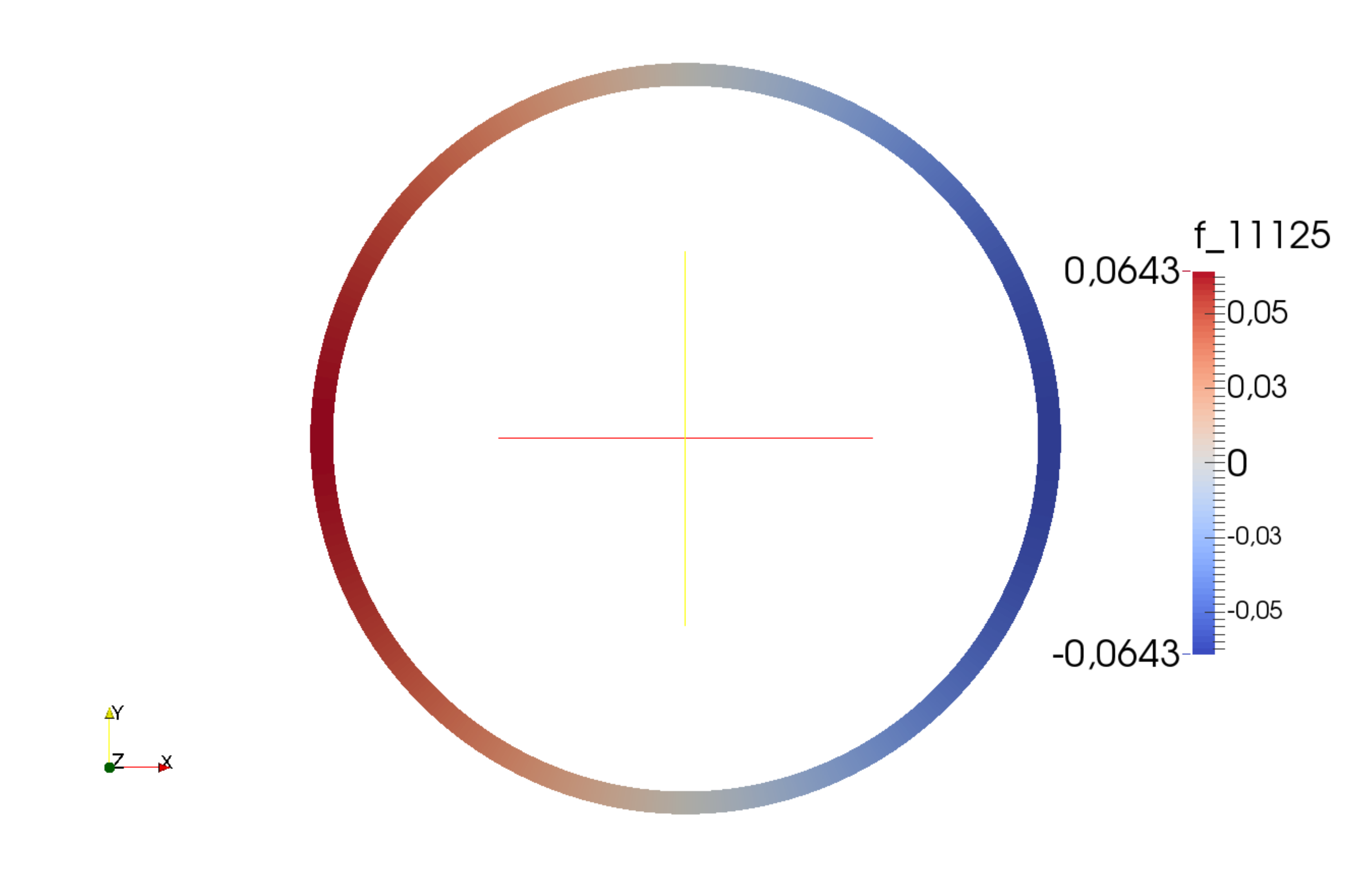} &
\includegraphics[scale=.15]{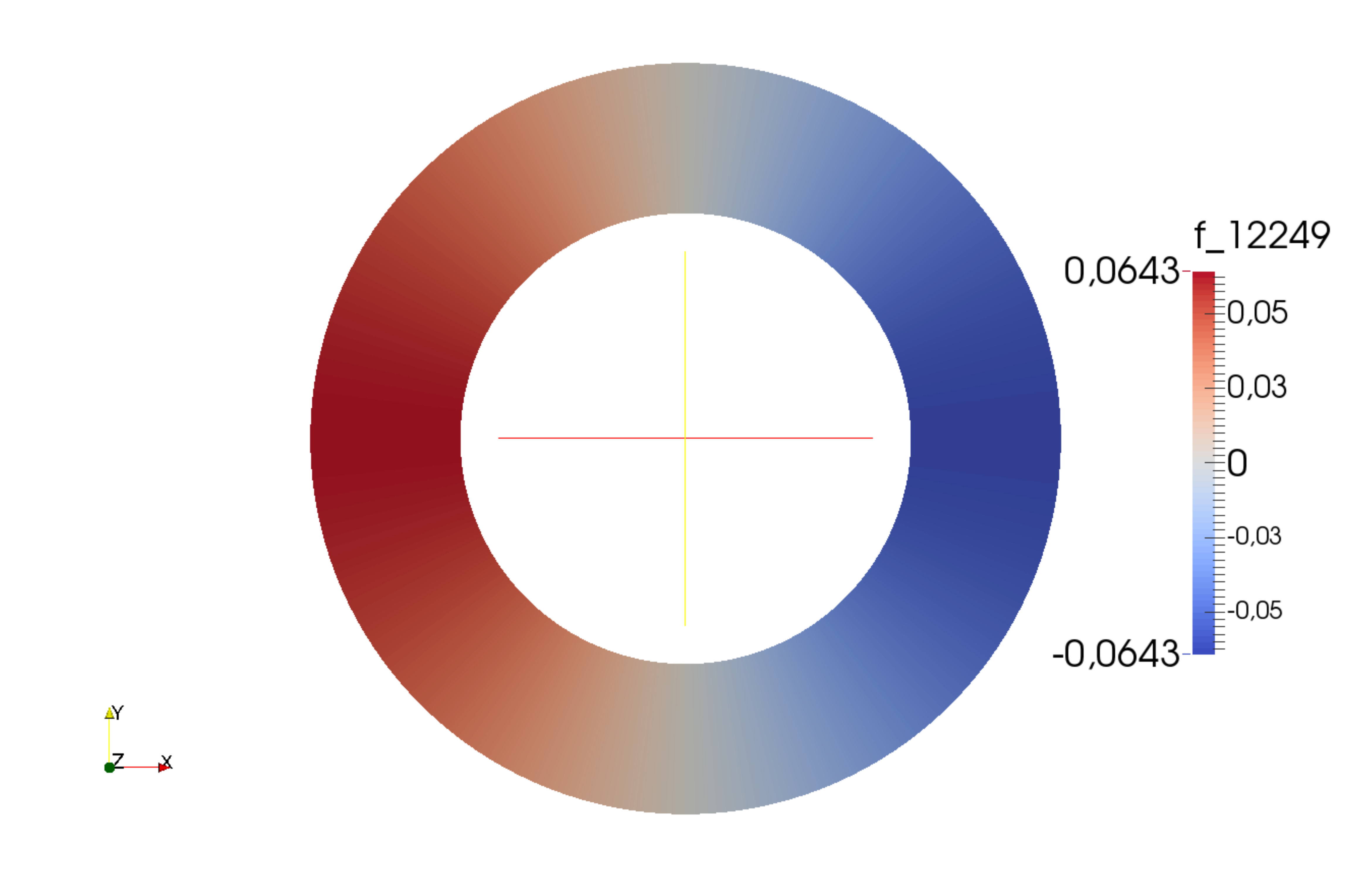} \\
(c)~~ Diffuse state $v_{\alpha,\delta}^\eps$ for $\eps = 0.03125 = \delta^{1/2}$. & (d)~~ Diffuse state $f_{\alpha,\delta}^\eps$ for $\eps = .25$. 
\end{tabular}}
\caption{A comparison of different state functions and the exact data in a). In a) and b), the state is only defined on $\partial B$, so we therefore applied the constant 
extension $E_B$ for the visualization, see Section \ref{sec:extensions}. 
In b), c) and d), $\delta =  2^{-10}$ and $\alpha = \delta/2$}\label{fig:compState}
\end{figure}

\begin{figure}
\centerline{%
\begin{tabular}{c@{\hspace{1pc}}c}
\includegraphics[scale=.15]{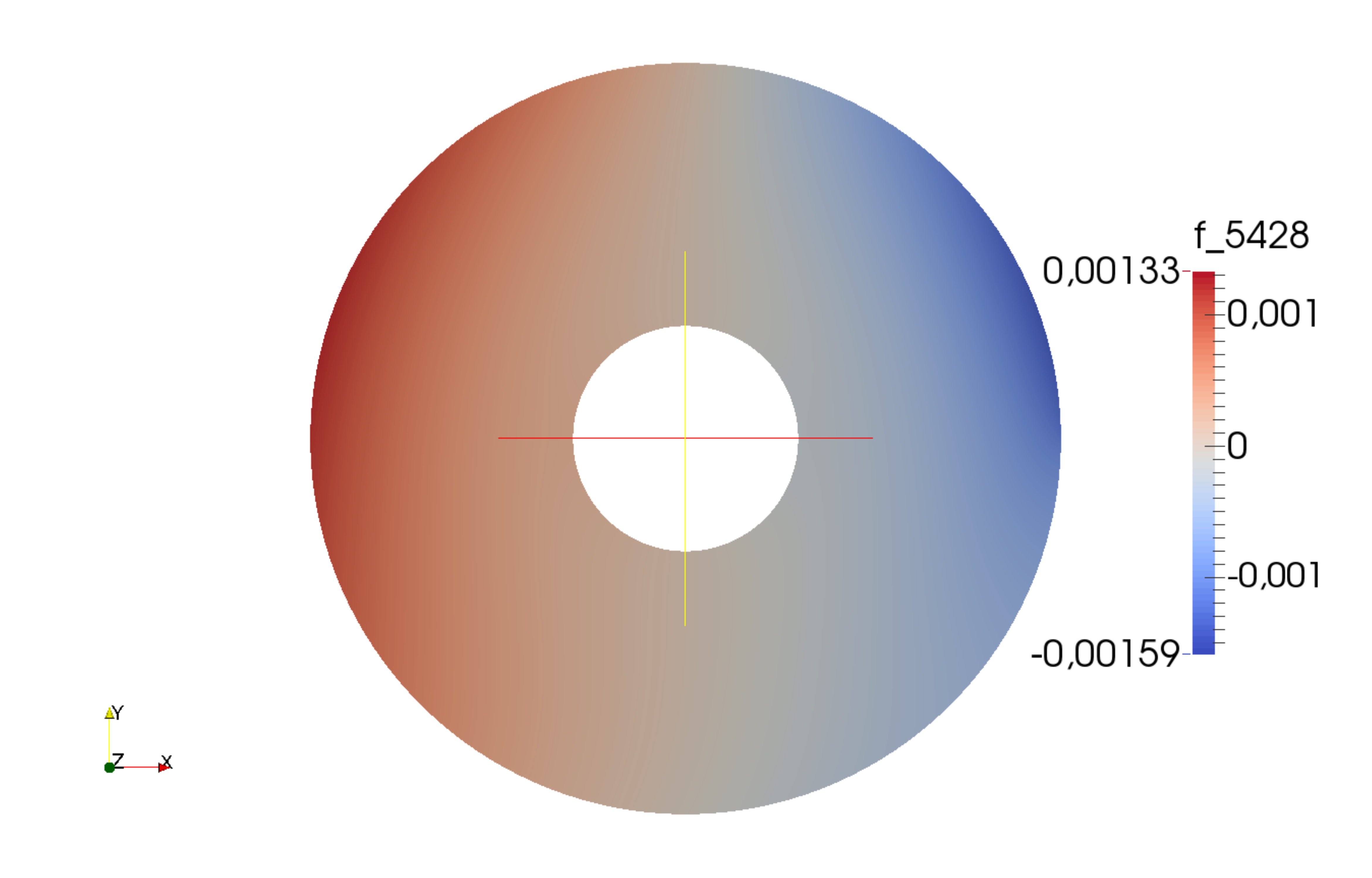} &
\includegraphics[scale=.15]{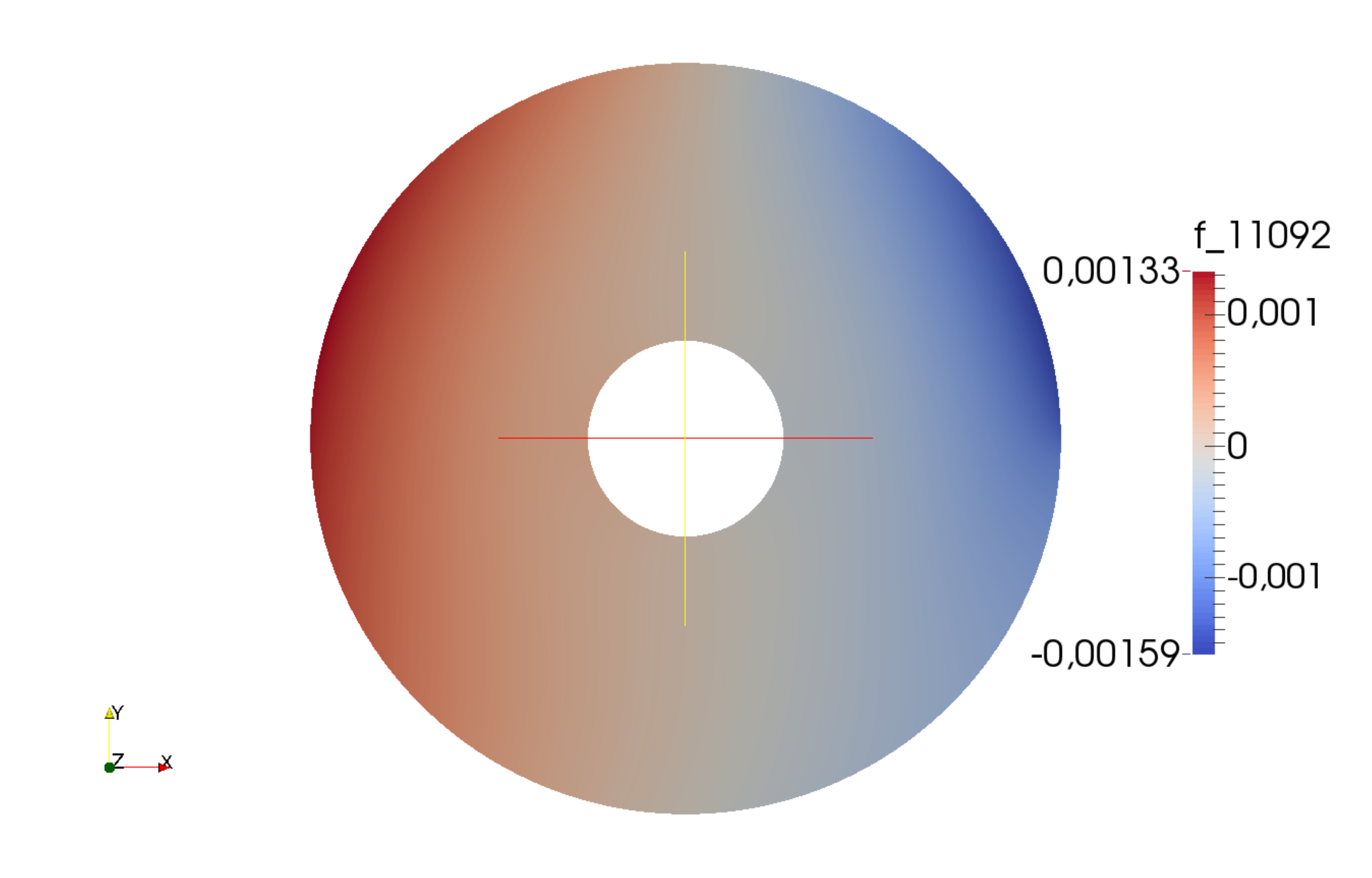} \\
(a)~~ The adjoint $p_{\alpha,\delta}$ on the exact mesh $D$. & (b)~~ The diffuse adjoint $p^\delta_{\alpha,\eps}$ for $\eps = .03125 = \delta^{1/2}$. 
\end{tabular}}
\caption{A comparison of the adjoint on the exact mesh and the diffuse mesh.
Here, $\delta =  2^{-10}$ and $\alpha = \delta/2.$}\label{fig:compAdj}
\end{figure}

The final issue we will investigate numerically is the convergence rates of
\begin{equation}
 \|u_{\alpha,\delta}^\eps - \tilde u^\dagger\|_{\mathcal{U}^\eps}, \nonumber
\end{equation}
for choices of $\alpha = C\delta^\mu$ and $\eps = c\delta^\nu$. 
In Figure \ref{fig:logplotHalf} we see convergence rates for the choice $\alpha = \delta/2$. In a), the convergence rate 
on the exact mesh is displayed. The rate seems, on average, to be of order $O(\delta^{1/2})$, but it is quite inconsistent
from step to step. 
This leads us to believe that a stronger source condition holds true and better convergence rates may be obtained, see Section~\ref{sec:rates}.
If the smoother source condition is satisfied, we can choose $\alpha = C\delta^{2/3}$. The convergence rates for this choice of 
$\alpha$ is displayed in Figure \ref{fig:logplotTThird}. In a), we now see a much more consistent convergence rate of order $O(\delta^{2/3})$. 

For the convergence rates associated with the diffuse domain method, we have more inconsistent rates. Generally, the convergence
rates can only be guaranteed for small choices of $\delta$ and $\eps$, and particularly the latter is difficult to handle numerically,
due to mesh limitations on standard computers. However, we see in Figure \ref{fig:logplotHalf}b) that the choices $\eps = \delta^{1/2}/4$
and $\eps = \delta^{2/3}/4$ yield roughly the same convergence rate, while $\eps = \delta^{1/3}/4$ yields a worse rate. 

For the case $\alpha = C\delta^{2/3}$, displayed in Figure \ref{fig:logplotTThird}, the numerics become more challenging. 
We observe from the rates associated with the inverse solutions on the exact mesh that we only obtain the theoretical convergence 
$\|u_{\alpha,\delta} - u^\dagger\|_{L^2(\partial H)} = O(\delta^{2/3})$ for small values of $\delta$. Hence, choosing $\eps = \delta^\nu$
might be numerically challenging for these values of $\delta$. 
However, the constant in Theorem \ref{thm:error} is not explicit, 
and we therefore select heuristically $C$ in $\eps=C\delta^{\nu}$. 
From Figure \ref{fig:logplotTThird}b), we observe that the choice $\eps = 35\delta^{2/3}$ yields a better rate than choosing $\eps = 10\delta^{1/2}$, which again yields a better rate than 
$\eps = 2.8\delta^{1/3}$. Furthermore, for the smallest noise values, the convergence rate associated with the choices
$\eps = 35\delta^{2/3}$ and $\eps = 10\delta^{1/2}$ actually seems to be of order $O(\delta^{2/3})$, which is the optimal
rate from standard theory, see \cite{EHN96}. The choice $\eps = C\delta^{1/2}$ is better than our theory suggests.
Roughly, this may be explained as follows. Measuring in a norm similar to a weighted $W^{1,1}$-norm gives approximations of order $\eps^2$ instead of $\eps^{3/2}$, see \cite{BES2014} and Theorem~\ref{thm:estimate}. Using this in Theorem~\ref{thm:error}, the optimal choice in Corollary~\ref{cor:convergence_rates} is actually $\eps\approx \delta^{1/2}$. As for coarse discretizations all norms are equivalent with moderate constants this may explain the observed behavior.

\begin{figure}
\centerline{%
\begin{tabular}{c@{\hspace{1pc}}c}
\includegraphics[scale=.4]{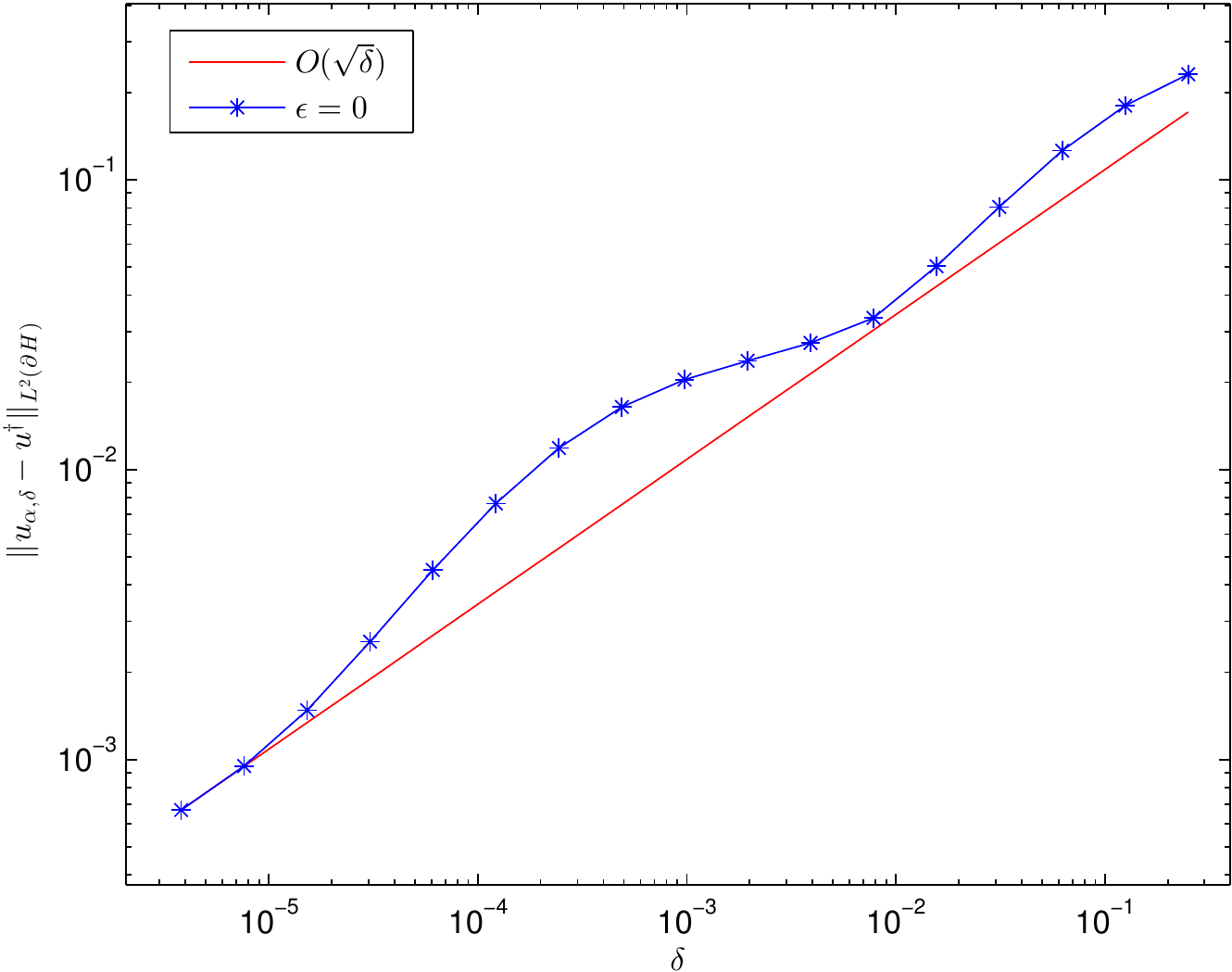} & \includegraphics[scale=.4]{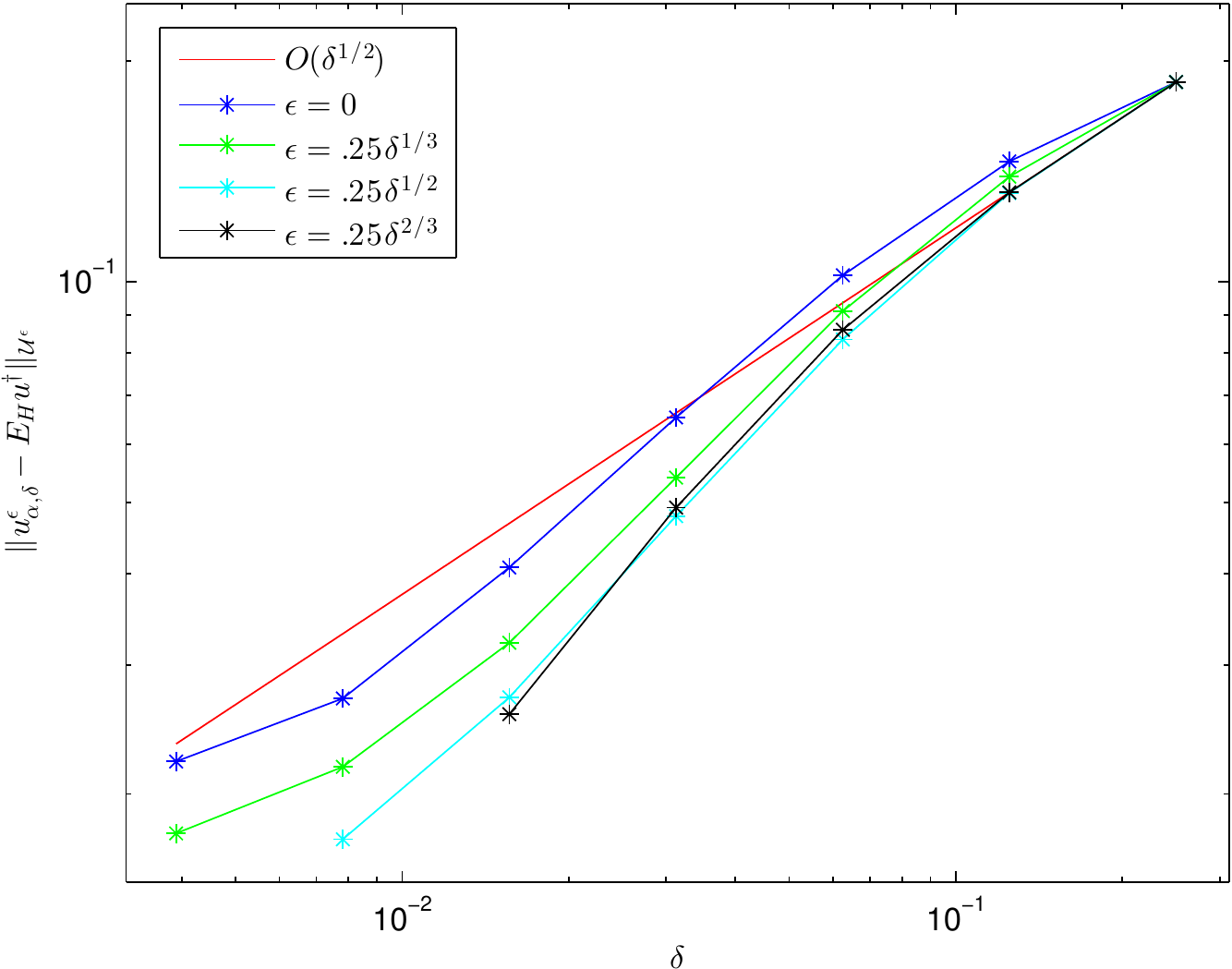} \\
(a)~~ Convergence rate for the control 
&
(b)~~ Convergence rate for the control \\ on the exact mesh $D$.   &on  the diffuse mesh $D_\eps$, with $\eps = .25\delta^\nu$,\\& where $\nu = \{1/3, 1/2, 2/3\}$. 
\end{tabular}}
\caption{A log-log plot of the convergence rates for different choices of diffuse domain parameter $\eps$. 
In both subplots we see the actual convergence rates (experimental), compared to the theoretical rate of order $O(\eps^{1/2})$.
Here, $\alpha = \delta/2$.}\label{fig:logplotHalf}
\end{figure}

\begin{figure}
\begin{subfigure}{0.4\textwidth}
\includegraphics[scale=.4]{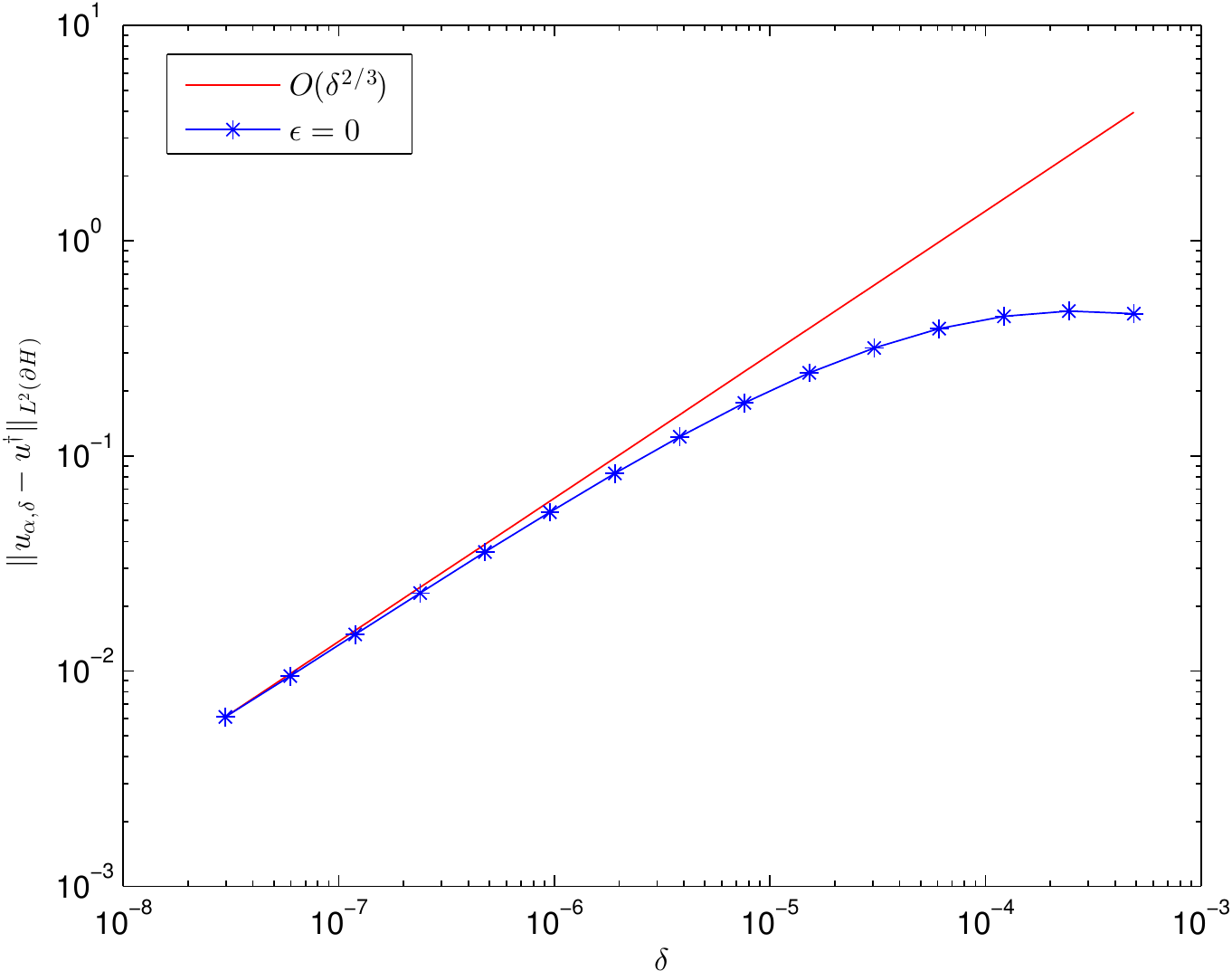}
\caption{Convergence rate for the control on the exact mesh $D$ for $\alpha=\delta^{2/3}$.}
\end{subfigure}\quad
\begin{subfigure}{0.4\textwidth}
\includegraphics[scale=.4]{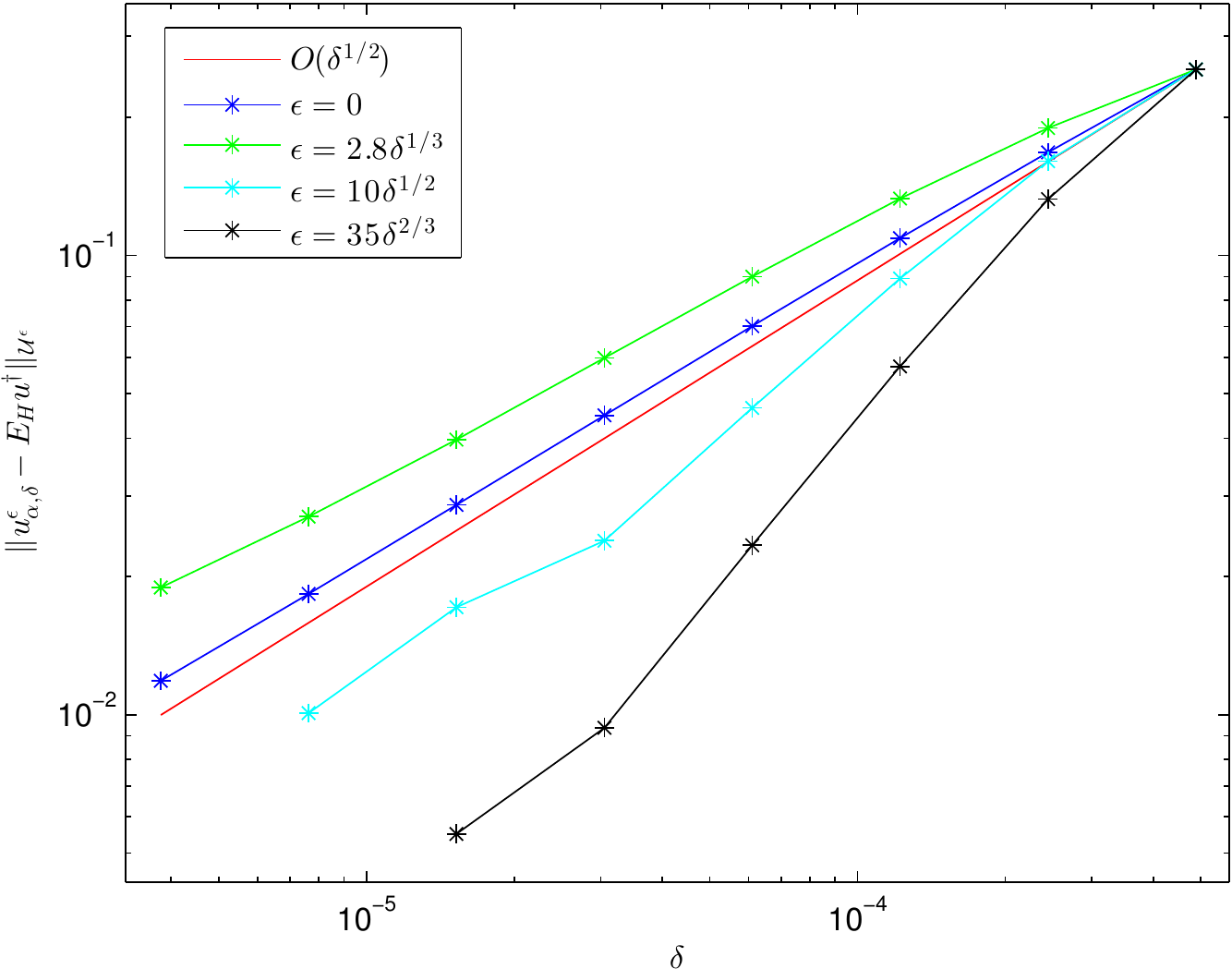}
\caption{Convergence rate for the control on the diffuse mesh $D_\eps$, with $\eps = C\delta^\nu$ for $\alpha=2\delta^{2/3}$.}
\end{subfigure}
\caption{A log-log plot of the convergence rates for different choices of diffuse domain parameter $\eps$.  
In both subplots we see the actual convergence rates (experimental), compared to the theoretical rate of order $O(\delta^{2/3})$.
All errors in (B) are scaled to be equal at the largest noise value. The notation $\eps = 0$ means computations on the exact mesh, i.e. as in (A).}\label{fig:logplotTThird}
\end{figure}

\section{Discussion and conclusions}
We applied a diffuse domain method to variational regularization methods.
This allowed us to handle complex geometries in a computationally efficient way. The additional error introduced by the diffuse domain method can be made arbitrarily small such that the overall error in the method is dominated by modeling errors and measurement noise. 
As a model problem we chose ECG inversion for which we could show that Tikhonov regularization is indeed a regularization method.
Extensions to other inverse problems governed by an elliptic partial differential equation of second order seem to be straightforward.
The main difference to standard Tikhonov regularization in Hilbert spaces, where simple operator perturbations can be handled in a straightforward manner, is the choice of topology which depends on $\eps$, the parameter in the diffuse domain method.
As this topology is weaker than the standard Hilbert space norm, we could show convergence in a dual norm only. 
A key ingredient for our convergence result is the reformulation of Tikhonov regularization as a constraint optimization problem, which gives additional control over the state, which in turn gave some compactness. Under the usual source conditions we could prove convergence rates in the stronger standard Hilbert space norm when an a priori parameter choice rule is used.
Using the methods present here, it should be possible to analyze also other parameter choice rules, and the use of nonlinear forward problems should also be feasible.
Extending the results of \cite{BES2014} to parabolic problems, one can also deal with time-dependent inverse problems. Here, the diffuse domain method is particularly suited when dealing with time-dependent geometries as e.g. a beating heart.
Another interesting point, which is not in the scope of this paper, is how errors in the distance function will influence the diffuse domain method.
On the continuous level noisy distance functions will lead to rough surfaces and new challenges come up.
Of particular interest is the case when only finitely many measurements, and hence measurement locations, are available, which makes it necessary to construct a distance function in a way that the noise is not dominant.

\section*{Acknowledgements}

MB and MS acknowledge support by ERC via Grant EU FP 7 - ERC Consolidator Grant 615216 LifeInverse. MB acknowledges support by the German Science Foundation DFG via  EXC 1003 Cells in Motion Cluster of Excellence, M\"unster, Germany.
OLE acknowledges support by DAAD for his one year research stay at WWU M\"unster. 

\bibliographystyle{abbrv}
\bibliography{biblio}

\begin{appendix}
\section{Basic Properties of Diffuse Approximations}
In this appendix we collect and extend some results of \cite{BES2014}.
We let $E$ be one of the extensions $E_B$ or $E_H$ defined in Section~\ref{sec:extensions} and $\gamma$ be one of the weighting functions $\gamma_B$ or $\gamma_H$, and assume that $\eps_0$ is sufficiently small. Moreover let $\Gamma=\partial D\cap {\rm supp}(\gamma)$. The constants $C$ are independent of $\eps$.
For $t\in (-\eps,\eps)$, we define the mapping $\Phi_t(x)=x+tn(x)$, $x\in\partial D$, and note that $\Phi_t(\partial D)=\{x\in\Omega: d_D(x)=t\}$. Moreover, cf. \cite[Eq. (9)]{BES2014},
\begin{align}\label{eq:det_to_one}
 \lim_{t \to 0}\sup_{x\in\Gamma} |\det D \Phi_t(x)-1-t \Delta d_D(x)| =0.
\end{align}
For any integrable $v$ the transformation formula implies
\begin{align}\label{eq:fubini}
   \int_{\Omega} v(x) |\nabla\omega^\eps| \gamma \d x = \frac{1}{2\eps}\int_{\Gamma}\int_{-\eps}^\eps v(x+tn(x)) |\det D\Phi_t(x)| \d t\d\sigma(x).
\end{align}

Let us begin with deriving some basic properties of the extensions constant off the interface defined in Section~\ref{sec:extensions}. 
\begin{lemma}\label{lem:extension}
There exists constant $c(\eps),C(\eps)>0$ such that for any $v\in L^2(\Gamma)$
\begin{align*}
  c(\eps) \|v\|_{L^2(\Gamma)} &\leq \|Ev\|_{L^2(\gamma |\nabla\omega^\eps|)}  \leq C(\eps) \|v\|_{L^2(\Gamma)}
\end{align*}
and $c(\eps)\to 1$ and $C(\eps)\to 1$ as $\eps \to 0$.
\end{lemma}
\begin{proof}
According to \eqref{eq:fubini} and $(Ev)(x+tn(x)) =v(x)$, $x\in\Gamma$, we have 
\begin{align*}
   \int_{\Omega} |E_B f(x)|^2 |\nabla\omega^\eps| \gamma_B \d x = \frac{1}{2\eps}\int_{\partial B} |f(x)|^2 \int_{-\eps}^\eps  \det D\Phi_t(x) \d t\d\sigma(x),
\end{align*}
and the assertion follows from \eqref{eq:det_to_one}.
\end{proof}

Lemma~\ref{lem:extension} implies that $E_B$ and $E_H$ are bounded, injective and have closed range.

The next issue, concerns the approximation of diffuse integrals. We set
\begin{align*}
 \Gamma_t = \{x\in\Omega: {\rm dist}(x,\Gamma)<t\}.
\end{align*}
The following is a central estimate.
\begin{thm}\label{thm:estimate} 
 Let $1\leq p < \infty$. There exists a constant $C>0$ such that
 
 (i) if $v\in W^{1,p}(\Omega;\omega^\eps)$, then
 \begin{align*}
  \|v\|_{L^p(\Gamma_\eps;|\nabla\omega^\eps|\gamma)}^p \leq C(\|v\|_{L^p(\Gamma)}^p + \eps^{p-1} \|\partial_n v\|_{L^p(\Gamma_\eps;\omega^\eps)}^p).
 \end{align*}
 (ii)  if $v\in W^{2,p}(\Omega;\omega^\eps)$, then
 \begin{align*}
  \|v\|_{L^p(\Gamma_\eps;|\nabla\omega^\eps|\gamma)}^p \leq C(\|v\|_{L^p(\Gamma)}^p + \eps^{p} \|\partial_n v\|_{L^p(\Gamma)}^p+\eps^{2p-1}\|\partial_n^2 v\|_{L^p(\Gamma_\eps;\omega^\eps)}^p).
 \end{align*}
\end{thm}
\begin{proof}
  (i) Using the basic inequality $(a+b)^p \leq 2^{p-1}(|a|^p + |b|^p)$, $a,b\in\RR$, we obtain by using the fundamental theorem of calculus and H\"olders inequality
  \begin{align*}
  |v(x+tn(x))|^p \leq 2^{p-1}( |v(x)|^p + |t|^{p-1} \int_{-|t|}^{|t|} |\partial_n v(x+sn(x))|^p \d s).
  \end{align*}
  Using the latter in \eqref{eq:fubini} and using \eqref{eq:det_to_one}, we obtain
\begin{align*}
  \|v\|_{L^p(\Omega;|\nabla\omega^\eps|\gamma)|)}^p \leq 2^{p-1}(\|v\|_{L^p(\Gamma)}^p+ \eps^{p-2}\int_\Gamma\int_0^\eps\int_{-t}^t |\partial_n v(\Phi_s(x))|^p\d s\d t\d\sigma).
\end{align*}
Using Fubini's theorem we further may write
\begin{align*}
 \frac{1}{\eps} \int_\Gamma\int_0^\eps\int_{-t}^t |\partial_n v(\Phi_s(x))|^p\d s\d t\d\sigma &\leq C\frac{1}{\eps}\int_0^\eps \int_{\Gamma_t} |\partial_n v(x)|^p\d x\d t.
\end{align*}
As in \cite[Section 5.1]{BES2014} using the transformation $s=-S(t/\eps)$, one completes the proof showing
$$
\frac{1}{\eps}\int_0^\eps \int_{\Gamma_t} |\partial_n v(x)|^p\d x\d t\leq \int_{\Gamma_\eps} |\partial_n v(x)|^p\omega^\eps\d x.
$$
(ii) Applying twice the fundamental theorem of calculus yields
\begin{align*}
 v(x+tn(x))=v(x) + t \partial_nv(x) + \int_0^t\int_0^s \partial_n^2v(x+rn(x))\d r\d s.
\end{align*}
The proof is then finished with similar arguments as in (i).
\end{proof}

With the usual modifications one shows that Theorem~\ref{thm:estimate} also holds for $p=\infty$.
We start with a diffuse trace lemma, cf.\@ \cite[Theorem~4.2]{BES2014}. We give a different proof.
\begin{lemma}\label{lem:diffusetrace}
There exists a constant $C >0$ such that
\begin{equation}
        \|v\|_{L^2(\gamma|\nabla \omega^\eps|)} \leq C \| v \|_{\H^\eps} \text{ for all } v \in \H^\eps.
\end{equation}
\end{lemma} 
\begin{proof}
 The usual trace theorem \cite{Adams1975} assures that $\|v\|_{L^2(\Gamma)}\leq C \|v\|_{H^{1}(D)}\leq C \|v\|_{\H^\eps}$. The result then follows from Theorem~\ref{thm:estimate} (i).
\end{proof}

Operator perturbations induced by the diffuse integrals can be treated using the following.
\begin{lemma}\label{lem:error_diffuse_integral}
  Let $1\leq  p \leq \infty$ and $v\in W^{k,p}(\Omega,\omega^\eps)$, $k\in \{0,1,2\}$. Then there exists a constant $C>0$ independent of $\eps$ such that
 
 (i) if $k\leq 1$ there holds
 \begin{align*}
  |\int_\Omega v \omega^\eps \d x - \int_D v \d x| \leq C \eps^{1+k-\frac{1}{p}} \|v\|_{W^{k,p}(\Omega;\omega^\eps)},
 \end{align*}
 (ii) if $k=1$, then
 \begin{align*}
  \|v-Ev\|_{L^p(|\nabla\omega^\eps|\gamma)}\leq C \eps^{1-\frac{1}{p}} \|v\|_{W^{1,p}(\Omega;\omega^\eps)},
 \end{align*}
 (iii) if $k=2$, then
 \begin{align*}
  \|v-Ev\|_{L^p(\gamma|\nabla\omega^\eps|)} \leq C (\eps \|\partial_n v\|_{L^p(\Gamma)}+\eps^{2-\frac{1}{p}}\|\partial_n^2 v\|_{L^p(\Gamma_\eps;\omega^\eps)}).
 \end{align*}
 (iv) if $k=2$, $v=0$ on $\Gamma$ and $w\in W^{1,2}(\Omega;\omega^\eps)$, then
 \begin{align*}
  | \int_\Omega vw |\nabla\omega^\eps|\gamma\d x| \leq C \eps^{\frac{3}{2}} \|v\|_{W^{2,2}(\Omega;\omega^\eps)} \|w\|_{W^{1,2}(\Omega;\omega^\eps)}.
 \end{align*}
\end{lemma}
\begin{proof}
 Assertions (i) and (iv) are proven in \cite[Theorem~5.1, Theorem~5.2, Theorem~5.6]{BES2014}. 
 To prove (ii) we apply Theorem~\ref{thm:estimate} (i) to $v-Ev$.
 As $v-Ev=0$ on $\Gamma$ and $\partial_n(v-Ev)=\partial_n v$, we obtain
 \begin{align*}
  \|v-Ev\|_{L^p(\Gamma_\eps;|\nabla\omega^\eps|\gamma)}\leq C \eps^{1-\frac{1}{p}} \|\partial_n v\|_{L^{p}(\Gamma_\eps;\omega^\eps)}.
 \end{align*}
 This yields the assertion. 
 (iii) is a direct consequence of Theorem~\ref{thm:estimate} (ii).
\end{proof}

A further tool in studying the diffuse domain method is the following lemma \cite[Lemma~4.9]{BES2014}.
\begin{lemma}[Poincar\'e-Friedrichs-type inequality]\label{lemma:diffusefriedrichs}
There exists a constant $C >0$ such that
\begin{equation}
        \Vert v \Vert_{\H^\eps}^2 \leq C (\|\nabla v\|^2_{L^2(\Omega;\omega^\eps)} + \|v\|^2_{L^2(\Omega;\gamma|\nabla\omega^\eps|)}) \quad\text{for all } v \in \H^\eps.
\end{equation}
\end{lemma} 

\end{appendix}
\end{document}